\documentclass{article}

\title{Intersection cohomology of rank two\\ character varieties of surface groups}
\author{Mirko Mauri}
\usepackage[utf8]{inputenc}
\usepackage[T1]{fontenc}
\usepackage{fixltx2e}
\usepackage{amsthm}
\usepackage{mathrsfs}

\usepackage{mathtools}
\usepackage{graphicx}
\usepackage{longtable}
\usepackage{float}
\usepackage{wrapfig}
\usepackage{rotating}
\usepackage{amsmath}
\usepackage{textcomp}
\usepackage{marvosym}
\usepackage{wasysym}
\usepackage{amssymb}
\usepackage{hyperref}
\usepackage{tikz}
\usepackage{tikz-cd}
\usepackage{float}
\usepackage{stmaryrd}
\usepackage{enumerate}
\usepackage{caption}
\usepackage{multirow}
\usepackage{fancyhdr}
\usepackage{enumitem}

\usetikzlibrary{matrix}
\usetikzlibrary{shapes.multipart}
\usetikzlibrary{arrows}
\usepackage{amsthm,amsmath,amsfonts,amscd}
\usepackage{tikz-3dplot}
\usetikzlibrary{decorations.markings}
\usetikzlibrary{arrows}
\usepackage{parskip}
\allowdisplaybreaks
\newcommand{\Z}{\mathbb{Z}}
\newcommand{\Aff}{\mathbb{A}}
\newcommand{\Perv}{\mathrm{Perv}}
\newcommand{\rat}{\mathrm{rat}}
\newcommand{\Gm}{\mathbb{G}_m}

\newcommand{\CC}{\mathbb{C}}
\newcommand{\QQ}{\mathbb{Q}}
\newcommand{\RR}{\mathbb{R}}

\newcommand{\ZZ}{\mathbb{Z}}
\newcommand{\PP}{\mathbb{P}}

\newcommand{\Spec}{\operatorname{Spec}}
\newcommand{\Gr}{\operatorname{Gr}}

\newcommand{\Gl}{\operatorname{GL}}
\newcommand{\Sl}{\operatorname{SL}}
\newcommand{\PGl}{\operatorname{PGL}}

\newcommand{\alb}{\operatorname{alb}}

\newcommand{\MB}{M_{\mathrm{B}}}
\newcommand{\MDol}{{M}_{\mathrm{Dol}}}
\newcommand{\TDol}{{T}_{\mathrm{Dol}}}
\newcommand{\Mukai}{M(S, v)}
\newcommand{\Kumm}{K(S, v)}
\newcommand{\Map}{\operatorname{Map}}
\newcommand{\Stab}{\operatorname{Stab}}
\newcommand{\Mod}{\mathrm{Mod}}

\newcommand{\MDolG}{{M}_{\mathrm{Dol}}(C,G)}
\newcommand{\MBG}{{M}_{\mathrm{B}}(C,G)}

\newcommand{\tMBG}{\widetilde{M}_{\mathrm{B}}(C,G)}
\newcommand{\tMDolG}{\widetilde{M}_{\mathrm{Dol}}(C,G)}

\newcommand{\MDolSl}{{M}_{\mathrm{Dol}}(C, \Sl_n)}

\newcommand{\Hom}{\operatorname{Hom}}

\newcommand{\Pic}{\operatorname{Pic}}
\newcommand{\Sing}{\operatorname{Sing}}
\newcommand{\Bl}{\operatorname{Bl}}
\newcommand{\Fix}{\operatorname{Fix}}
\newcommand{\TJac}{T^*\operatorname{Jac}(C)}
\newcommand{\Jac}{\operatorname{Jac}(C)}

\newcommand{\Core}{\mathrm{Core}}

\newcommand{\MBtw}{M^{\mathrm{tw}}_{\mathrm{B}}}
\newcommand{\MBTW}{M^{\mathrm{tw}}_{\mathrm{B}}(C,G,d)}
\newcommand{\MDolTW}{M^{\mathrm{tw}}_{\mathrm{Dol}}(C,G,d)}
\newcommand{\MDoltw}{M^{\mathrm{tw}}_{\mathrm{Dol}}}

\newcommand{\Def}{\mathrm{Def}}

\usepackage{verbatim}

\begingroup
\makeatletter
\@for\theoremstyle:=definition,remark,plain\do{%
\expandafter\g@addto@macro\csname th@\theoremstyle\endcsname{%
\addtolength\thm@preskip\parskip
}%
}
\endgroup
\usepackage{graphicx}
\usepackage[capitalise]{cleveref}
\newtheorem{thm}{Theorem}[section]
\newtheorem{lem}[thm]{Lemma}
\newtheorem{cor}[thm]{Corollary}
\newtheorem{defn}[thm]{Definition}
\newtheorem{prop}[thm]{Proposition}
\newtheorem{conj}[thm]{Conjecture}

\theoremstyle{definition}
\newtheorem{notation}[thm]{Notation}
\newtheorem{exa}[thm]{Example}
\newtheorem{rmk}[thm]{Remark}
\crefname{thm}{Theorem}{Theorems}
\Crefname{thm}{Theorem}{Theorems}
\Crefname{thm}{Theorem}{Theorems}
\Crefname{thm}{Theorem}{Theorems}
\crefname{lem}{Lemma}{Lemmas}
\Crefname{lem}{Lemma}{Lemmas}
\crefname{Conjecture}{Conjecture}{Conjectures}
\Crefname{Conjecture}{Conjecture}{Conjectures}
\crefname{Corollary}{Corollary}{Corollaries}
\Crefname{Corollary}{Corollary}{Corollaries}
\crefname{Claim}{Claim}{Claims}
\Crefname{Claim}{Claim}{Claims}
\crefname{Proposition}{Proposition}{Propositions}
\Crefname{Proposition}{Proposition}{Propositions}
\crefname{Remark}{Remark}{Remarks}
\Crefname{Remark}{Remark}{Remarks}
\crefname{Definition}{Definition}{Definitions}
\Crefname{Definition}{Definition}{Definitions}
\crefname{Example}{Example}{Examples}
\Crefname{Example}{Example}{Examples}
\crefname{Exercise}{Exercise}{Exercises}
\Crefname{Exercise}{Exercise}{Exercises}

\newtheoremstyle{plain2}    
   {}            
   {}            
   {\itshape}    
   {}            
   {\bfseries}   
   {.}           
   {5pt plus 1pt minus 1pt}  
   {{\thmnumber{#1} \thmname{#2}{\thmnote{ (#3)}}}}          

\begin{document}
\maketitle
\begin{abstract}
    For $G = \mathrm{GL}_2, \mathrm{SL}_2, \mathrm{PGL}_2$ we compute the intersection E-polynomials and the intersection Poincar\'{e} polynomials of the $G$-character variety of a compact Riemann surface $C$ and of the moduli space of $G$-Higgs bundles on $C$ of degree zero. We derive several results concerning the P=W conjectures for these singular moduli spaces.
\end{abstract}
\vspace{0.5 cm}
\tableofcontents

\section{Introduction}
Let $C$ be a compact Riemann surface of genus $g \geq 2$, and $G$ be a complex reductive algebraic group.  The $G$-character variety of $C$, or \textbf{Betti moduli space}, is the affine GIT quotient
\begin{align} \MBG \coloneqq & \Hom(\pi_1(C), G)\sslash G \nonumber \\
 = & \big\{ (A_1, B_1, \ldots, A_{g}, B_{g}) \in G^{2g} \, \big| \, \prod^{g}_{j=1}[A_j, B_j]=1_{G} \big\}\sslash G. \label{eq:presentation}
\end{align}
It parametrises isomorphism classes of semi-simple representations of the fundamental group of $C$ with value in $G$. Remarkably, the space $\MBG$ is homeomorphic to the \textbf{Dolbeault moduli space} $\MDolG$, parametrising instead isomorphism classes of semistable principal $G$-Higgs bundles on $C$ of degree zero
; see \cite{Simpson1994}. 
For example, a $\Gl_n$-Higgs bundle is a pair $(E, \phi)$ with $E$ vector bundle of rank $n$ and $\phi \in H^0(C, \operatorname{End}(E)\otimes K)$. Such a pair is an $\Sl_n$-Higgs bundle, if in addition the determinant of $E$ is trivial and the trace of $\phi$ vanishes.

Since the pioneering work of Hitchin \cite{Hitchin1987}, the (non-algebraic) homeomorphism between $\MBG$ and $\MDolG$ have been exploited to study the topology of character varieties. The main result of this paper is the computation of some geometric invariants of $\MDolG$ and $\MBG$, namely the intersection E-polynomials and the intersection Poincar\'{e} polynomials of $\MDolG$ and $\MBG$ for $G=\Gl_2, \Sl_2, \PGl_2$; see \S\ref{sec:computations}. 

The motivation for this work stems from the newly stated P=W conjectures \cite[Conjecture 1.2, 1.4, 1.5]{FelisettiMauri2020} for singular character varieties; see also \cite[Question 4.1.7]{deCataldoMaulik2018} and the seminal paper \cite{deCataldoHauselMigliorini2012}. In fact, the explicit knowledge of intersection E-polynomials and intersection Poincar\'{e} polynomials is an essential ingredient in the proof of the P=W conjectures in rank 2 and genus 2 in \cite[Main Theorem]{FelisettiMauri2020}. Here, as an applications of our computations, we collect in \cref{thm:applications} several results related to the P=W conjectures in rank 2 and arbitrary genus. For brevity, we simply write $\MB$ for $\MBG$ and $\MDol$ for $\MDolG$ when $G=\Gl_2$ or $\Sl_2$, and we suppress subscripts ${}_\text{B}$ or ${}_\text{Dol}$ when we refer indifferently to the Betti or the Dolbeault side. \newpage

\begin{thm}[Remarks on the P=W conjectures]\label{thm:applications} The following facts hold:
\begin{enumerate}[label=\Alph*.]
\item the intersection E-polynomial $IE(\MB)$ is palindromic;
\item the PI=WI conj. for $\Sl_2$ is equivalent to the PI=WI conj. for $\Gl_2$;
\item $IH^{< 4g-6}(M)$ is freely generated by tautological classes;
\item the mixed Hodge structure on $H^*(\MDol)$ is not pure for $g>3$;
\item the P=W conjecture for any resolution of $M(C, \Gl_n)$ fails, when $M(C, \Gl_n)$ does not admit a symplectic resolution\footnote{A resolution of singularities $f\colon X\to Y$ is symplectic if a
holomorphic symplectic form on the smooth locus of $X$
extends to a symplectic form on $Y$.}.
\end{enumerate}
\end{thm}
In \S \ref{sec:P=Wconjectures} we recall the content of the P=W conjectures, and we give a proof of \cref{thm:applications}. Here we briefly explain the relevance of the previous statements in view of the P=W conjectures. 
\begin{enumerate}[label=\Alph*.]
    \item \cref{thm:applications}.A (\cref{thm:palindromicity}) provides a numerical evidence for the PI=WI conjecture. Indeed, the PI=WI conjecture 
    implies the palindromicity of $IE(M_B)$.
    \item \cref{thm:applications}.B (\cref{cor:PSL2givesGL2vv}) is a useful reduction statement. It says that it is enough to prove the PI=WI conjecture only for a portion of the intersection cohomology, namely its $\Gamma$-invariant part; see \eqref{eq:GLtoSL} and \eqref{eq:PGLtoSL}.
    \item The known proofs of the P=W conjecture for twisted character varieties \cite{deCataldoHauselMigliorini2012} and \cite{deCataldoMaulikShen2019} (cf also \S\ref{sec:twistedcharacter}) rely on the generation by tautological classes of the $\Gamma$-invariant part of $H^*(M)$. This is unknown for the intersection cohomology of the singular moduli spaces. \cref{thm:applications}.C (\cref{thm:generation}) provides a partial answer, i.e.\ the tautological generation of the intersection cohomology in low degree. 
    \item[D\&E.]  \cref{thm:applications}.D and \ref{thm:applications}.E stress the difference between the P=W conjectures for character varieties with or without a symplectic resolution; see \S \ref{P=WPI=WI}, \S \ref{sec:nosymplecticresolution}, and also \cite{FelisettiMauri2020}. 
\end{enumerate}

Our strategy to compute the intersection E-polynomials of $M$ is to use the Kirwan--O'Grady desingularization $\pi_T\colon T \to M$ (\S \ref{sec:desingularization}), and determine all the summands of the decomposition theorem for $\pi_T$; cf \cite[Remark 2.28]{Kirwan87}. This is a subtle task that we can complete thanks to a tight control of the geometry of $\pi_T$.

\begin{thm}[Decomposition theorem for $\pi_T$]\label{introthm:decompositiontheorempiT} There is an isomorphism in $D^bMHM_{\mathrm{alg}}(M)$ or in $D^b(M)$ (ignoring the Tate shifts):
\begin{align*}
R\pi_{T, *}\QQ[\dim T]= IC_{M} & \oplus  \bigoplus^{2g-4}_{i=-2g+4}\QQ_{\Sigma}^{\lceil \frac{2g-3-|i|}{2} \rceil}[\dim \Sigma -2i](-2g+3-i) \\
& \oplus \bigoplus^{2g-4}_{i=-2g+4} i_{\Sigma^\circ,*}\mathscr{L}^{\lfloor \frac{2g-3-|i|}{2} \rfloor}[\dim \Sigma-2i](-2g+3-i)\\
& \oplus \bigoplus^{3g-4}_{j=-3g+4} \QQ^{b(j)}_{\Omega}[\dim \Omega-2j](-3g+3-j),    
\end{align*}
where 
\begin{itemize}
    \item $\Sigma$ is the singular locus of $M$;
    \item $\Omega$ is the singular locus of $\Sigma$;
    \item $i_{\Sigma^\circ}\colon \Sigma^\circ \coloneqq \Sigma \setminus \Omega \to \Sigma$ is the natural inclusion;
    \item  $\mathscr{L}$ is the rank-one local system on $\Sigma^{\circ}$ corresponding to a quasi-\'{e}tale double cover $q\colon \Sigma_{\iota} \to \Sigma$ branched along $\Omega$ (see \cref{def:involution});
    \item $ b(-3g+3+j)$ is the coefficient of the monomial $q^j$ in the polynomial  \begin{align*}
    \frac{\left(1-q^{2g-2}\right) \left(1-q^{2g}\right) \left(1-q^4 -q^{2
 g-3}-q^{2g-1}+2q^{2g}\right)}{\left(1-q\right)^3
 \left(1-q^2\right)}-\frac{1-q^{2g}}{1-q^2}\\
  -\frac{q (1 - q^{2 g - 3}) (1 - q^{2 g - 2})}{(1-q) (1-q^2)}.
\end{align*}
\end{itemize}
\end{thm}

Note in addition that the same decomposition holds for the Mukai moduli space of semistable sheaves on K3 or abelian surfaces with Mukai vector $v = 2 w \in H^*_{\mathrm{alg}}(S, \ZZ)$, where $w$ is primitive and $w^2=2(g-1)$, thus suggesting other potential applications of \cref{introthm:decompositiontheorempiT}. This is indeed a consequence of the \textbf{stable isosingularity principle} (\cref{Stableisosingularityprinciple}), which roughly says that Betti, Dolbeault and Mukai moduli spaces have the same type of singularities, in the sense of \cref{defn:equisingularity}.  

It is conceivable that the computation of the intersection E-polynomials in rank 3 can be pursued with no substantial conceptual difference. In higher rank however this seems hard. In fact, closed formulas may be cumbersome and less enlightening. Notwithstanding, we believe that the rank 2 case can inspire the investigation of the higher rank case, especially in relation to the P=W conjectures \cite{FelisettiMauri2020} and the Hausel--Thaddeus topological mirror symmetry conjecture for singular character varieties \cite[Remark 3.30]{Hausel13}. 

\subsection{Notation} 
The intersection cohomology of a complex variety $X$ with middle perversity and rational coefficients is denoted by $IH^*(X)$. Ordinary singular cohomology with rational coefficients is denoted by $H^*(X)$. The subscript $c$ stands for compactly supported intersection or ordinary cohomology, respectively $IH_c^*(X)$ and $H_c^*(X)$. Recall that they all carry mixed Hodge structures.

The Poincar\'{e} polynomial, the intersection Poincar\'{e} polynomial, the intersection Euler characteristic, the E-polynomial and the intersection E-polynomial are defined by 
\begin{align*}
    P_t(X) & = \sum_{d} \dim H^d(X) t^d, \\
    IP_t(X) & = \sum_{d} \dim IH^d(X) t^d, \\
    I\chi(X) & = \sum_{d} (-1)^d \dim IH^d(X),\\
    E(X) & = \sum_{r,s,d}(-1)^d \dim ( \Gr^W_{r+s} H^d_{c}(X, \CC))^{r,s} u^r v^s,\\
    IE(X) & = \sum_{r,s,d}(-1)^d \dim ( \Gr^W_{r+s} IH^d_{c}(X, \CC))^{r,s} u^rv^s.
\end{align*}
We will often write $q \coloneqq uv$. 

The action of a finite group $\Gamma$ on $X$ induces the splitting 
\[H^d(X) = H^d(X)^{\Gamma} \oplus H_{\mathrm{var}}^*(X),\]
where $H^d(X)^{\Gamma}$ is fixed by the action of $\Gamma$, and $H_{\mathrm{var}}^d(X)$ is the {variant} part, i.e.\ the unique $\Gamma$-invariant complement of $H^d(X)^{\Gamma}$ in $H^d(X)$. Analogous splittings hold for the $\Gamma$-modules $H^d_c(X)$, $IH^d(X)$ and $IH^d_c(X)$. The label ${}^{\Gamma}$ or ${}_{\mathrm{var}}$, written after the polynomials above, imposes to replace ordinary (or intersection) cohomology with its $\Gamma$-invariant or $\Gamma$-variant part respectively, e.g.\ $IE(X)^{\Gamma} = \sum_{r,s,d}(-1)^d \dim ( \Gr^W_{r+s} IH^d_{c}(X, \CC)^{\Gamma})^{r,s} u^rv^s$.

If $\iota\colon X \to X$ is an involution, we simply use  the superscript $+$ or $-$ to denote the $\iota$-invariant and $\iota$-variant part, e.g.\ $P(X)^{+} = \sum_d \dim H^d(X)^{+} t^d \coloneqq P(X)^{\langle \iota \rangle}$.


We always denote by $C$ a complex projective curve of genus $g \geq 2$, unless differently stated. For notational convenience, we simply write $\MB$ for $\MBG$ and $\MDol$ for $\MDolG$ when $G=\Gl_2$ or $\Sl_2$, and we suppress subscripts ${}_\text{B}$ or ${}_\text{Dol}$ when we refer indifferently to the Betti or the Dolbeault side. We adopt the same convention for the strata $\Sigma_{\mathrm{B}}(C,G)$, $\Sigma_{\mathrm{Dol}}(C,G)$, $\Omega_{\mathrm{B}}(C,G)$, $\Omega_{\mathrm{Dol}}(C,G)$.

\subsection{Computations}\label{sec:computations}
As an application of the decomposition theorem (\cref{introthm:decompositiontheorempiT}), we can express $IE(M)$ as a function of the E-polynomials of $M$, $\Sigma_{\iota}$ and $\Omega$; see \cref{prop:singulMB} for the definition of these strata, and \S \ref{sec:proofs} for the proofs of the following expressions.
\begin{thm}\label{thm:IEmainformula}
\begin{equation}\label{eq:IE}
 IE(M) =  E(M)+ (q^2E(\Sigma_{\iota})^+ + qE(\Sigma_{\iota})^-) \cdot \frac{1 - q^{2 g - 4}}{ 1-q^2} +E(\Omega) \cdot q^{2g-2}.   
\end{equation}
\end{thm}

\begin{thm}\label{thm:IEB} 
The intersection E-polynomials of $\MB$ are
\begin{align*}
IE(\MB(C, \Sl_2)) = & (q^{2g-2}+1)(q^2-1)^{2g-2}+\frac{1}{2}q^{2g-3}(q^2+1)((q+1)^{2g-2}\\ & -(q-1)^{2g-2}) +2^{2g-1}q^{2g-2}((q+1)^{2g-2}+(q-1)^{2g-2}),\\
IE(\MB(C, \PGl_2))  = &  (q^{2g-2}+1)(q^2-1)^{2g-2} + \frac{1}{2}q^{2 g - 3}(q^2+q+1)(q + 1)^{2 g - 2} \\
&   - 
 \frac{1}{2}q^{2 g - 3}(q^2-q+1)(q - 1)^{2 g - 2},\\
 IE(\MB(C, \Gl_2))  = & (q-1)^{2g} \cdot  IE(\MB(C, \PGl_2)).\\
\end{align*}
\end{thm}
\begin{cor}\label{cor:palindromic}
The intersection E-polynomials $IE(\MB)$ are palindromic.
\end{cor}
\begin{cor}
The intersection Euler characteristics of $M$ are
\begin{align*}
    I\chi(M(C, \Sl_2)) & =2^{2g-2}(2^{2g-1}+1),\\
    I\chi(M(C, \PGl_2)) & =3 \cdot 2^{2g-3}.
\end{align*}
\end{cor}
We list the intersection E-polynomial of $\MB(C, \Sl_2)$ in low genus (we truncate the polynomial at degree $3g-3$: the coefficients of the monomials of higher degree can be determined by symmetry, since $IE(\MB(C, \Sl_2))$ is palindromic of degree $6g-6$):
\begin{align*}
    g=2: & \quad 1+17q^2+\cdots\\
    g=3: & \quad 1-4q^2+75q^4+384q^6+\cdots\\
    g=4: & \quad 1-6q^2+15q^4+243q^6+3875q^8+\cdots\\
    g=5: & \quad 1-8q^2+28q^4-56q^6+1103q^8+28672q^{10}+71848q^{12}+\cdots.
\end{align*} 
\begin{rmk}
The intersection E-polynomial of $\MB(C, \Sl_2)$ is a polynomial in $q^2$. This fails for twisted $\Sl_2$-character varieties (cf \cite[(2)]{Mereb2015}; see also \S\ref{sec:twistedcharacter} for the definition of the twist), but it holds true for twisted $\PGl_2$-character varieties, since their cohomology is generated by classes of weight $4$; see \cite[Proposition 4.1.8]{HauselVillegas15}. The E-polynomial of $\MB(C, \Sl_2)$ is a polynomial in $q^2$ too; see \cite[Theorem 2]{MartinezMunoz16} or \cite[Theorem 1.3]{BaragliaHekmati17}.
\end{rmk}

\begin{thm}\label{thm:IEDol} The intersection E-polynomial of $\MDol(C, \Sl_2)$ is
\begin{align*}
    IE( \MDol(C, \Sl_2))  & = E( \MDol(C, \Sl_2)^{\mathrm{sm}}) + \frac{1}{2}(uv)^g \left((1-u)^g(1-v)^g+(1+u)^g(1+v)^g\right)\\
    & + \frac{1}{2}(uv)^{g+1}(1-(uv)^{2g-4})\left(\frac{(1-u)^g(1-v)^g}{1-uv}-\frac{(1+u)^g(1+v)^g}{1+uv}\right)\\
    & +2^{2g}(uv)^{2g-2}.
\end{align*}
\end{thm}
An explicit formula for the E-polynomial of the smooth locus of $\MDol(C, \Sl_2)$ was computed in \cite[Theorem 3.7]{KiemYoo08}. Together with \cref{thm:IEDol}, this gives the intersection Poincar\'{e} polynomial of $M(C, \Sl_2)$.
\begin{thm}\label{thm:IP} The intersection Poincar\'{e} polynomial of $M(C, \Sl_2)$ is
\begin{align*}
     IP_t(M(C, \Sl_2)) & = \frac{(t^3+1)^{2g}}{(t^2-1)(t^4-1)} + (g-1)t^{4g-3}\frac{(t+1)^{2g-2}}{t-1}\\ 
     & - \frac{t^{4g-4}}{4(t^2-1)(t^4-1)}\big( (t^2+1)^2(t+1)^{2g}-(t+1)^4(t-1)^{2g}
     \big)\\
     & + \frac{1}{2} t^{4g-4}((t+1)^{2g-2}-(t-1)^{2g-2})- \frac{1}{2} t^{4g-6}((t+1)^{2g}-(t-1)^{2g})\\
     & + \frac{1}{2}(2^{2g} -1) t^{4g-4}((t+1)^{2g-2}+(t-1)^{2g-2}).
\end{align*}

Therefore, we have
\begin{equation}\label{TaylorIPt}
    IP_t(M(C, \Sl_2)) = \frac{(t^3+1)^{2g}}{(t^2-1)(t^4-1)}-2g \cdot t^{4g-5}+O(t^{4g-4}).
\end{equation}
\end{thm}

We list the intersection Poincar\'{e} polynomial of $\MB(C, \Sl_2)$ in low genus:
\begin{align*}
    g=2: \quad & 1+t^2+17 t^4+17t^6\\
    g=3: \quad & 1+t^2+6 t^3+2 t^4+6 t^5+17 t^6+6 t^7+81 t^8+12 t^9+396 t^{10}+6t^{11}+66 t^{12}\\
    g=4: \quad & 1+t^2+8 t^3+2 t^4+8 t^5+30 t^6+16 t^7+31 t^8+72 t^9+59 t^{10}+72
t^{11}+385 t^{12}\\
&+80 t^{13}+3955 t^{14}+80 t^{15}+3885 t^{16}+16
t^{17}+259t^{18}\\
    g=5: \quad & 1+t^2+10 t^3+2 t^4+10 t^5+47 t^6+20 t^7+48 t^8+140 t^9+93 t^{10}+150t^{11}\\
    & +304 t^{12} +270 t^{13}+349 t^{14}+522 t^{15}+1583 t^{16}+532t^{17}+29414 t^{18}+532 t^{19}\\
    & +72170 t^{20}+280 t^{21}+28784t^{22}+30 t^{23}+1028 t^{24}.
\end{align*}

The Poincar\'{e} polynomial of $\MB(C, \Sl_2)$ was obtained in \cite[Theorem 1.5]{DaskalopoulosWentworth10}, which however contains small transcription errors (cf \cite[Theorem 2.2]{CappellLeeMiller00} and \cite[(47)]{DWWW11}).
\begin{align*}
    P_t(M(C, \Sl_2))& = \frac{(t^3+1)^{2 g}}{(t^2-1) (t^4-1)} + \frac{(t+1)^{2 g} (t^2+1) + (t-1)^{2 g} (t^2-1)}{2 (  t^4-1)}\\
    & + \sum^g_{k=2} \left\{ \binom{2g}{k} - \binom{2g}{k-2} \right\} t^{k + 2 \Mod[k, 2]} \frac{(t^{2 k - 2 \Mod[k, 2]}-1) ( 
     t^{2 g - 2 k + 2}-1)}{(t-1) (t^4-1)}\\
     & - \frac{1}{2} t ((t+1)^{2 g} + (t-1)^{2 g})+\frac{t^{2 g + 2}-1}{t-1} -t^{4 g - 4} + \frac{(t-1)^{2 g}t^{4 g - 4}}{4 (t^2+1)} \\
     & - \frac{(t+1)^{2 g} t^{4 g - 4}}{2 (t^2-1)}\bigg(\frac{2 g}{t+1} + \frac{1}{t^2-1} - 
  \frac{1}{2} + 3 - 2 g \bigg)\\
   & + \frac{1}{2}(2^{2 g} - 1)
  t^{4 g - 4} ((t+1)^{2 g - 2} + (t-1)^{2 g - 2} - 2),
\end{align*}
where $\Mod[k, l]$ is the reminder on division of $k$ by $l$. We can then inspect the difference $IP_t(M(C, \Sl_2))-P_t(M(C, \Sl_2))$ in low genus: 
\begin{align*}
    g=2: \quad & 16 t^4\\
    g=3: \quad & 6 t^3+t^4+6 t^5+t^6+6 t^7+79 t^8+t^{10}\\
    g=4: \quad & 8 t^3+t^4+8 t^5-20 t^6+16 t^7-19 t^8+22 t^9+56 t^{10}+56 t^{11}+327 t^{12}\\
    & +8 t^{13}+28 t^{14}+t^{16}\\
    g=5: \quad & 10 t^3 + t^4 + 10 t^5 - 65 t^6 + 20 t^7 - 196 t^8 - 35 t^9 - 
 20 t^{10} - 25 t^{11} + 124 t^{12} \\
 & + 240 t^{13} + 256 t^{14} + 262 t^{15} + 1279 t^{16} + 120 t^{17} + 211 t^{18} + 10 t^{19} + 45 t^{20} + t^{22}.
\end{align*}

\begin{cor}\label{cor:IP-P} Let $C$ be a curve of genus $g \geq 6$. Then we have
 \[IP_t(M)-P_t(M) = 2g \cdot t^3 +t^4 +2g \cdot t^5 - \left\lbrace \binom{2g}{3}-\binom{2g}{2}-2g\right\rbrace t^6 + O(t^7). \]
\end{cor}

At this point it is worthy to recall how the (intersection) cohomology of $M(C, G)$ with $G=\Gl_n, \Sl_n$ and $\PGl_n$ compare, and how to extend the previous results for the $\Sl_2$ case to $\PGl_2$ and $\Gl_2$.

The morphism 
\begin{equation}\label{albmap}
    \alb\colon  M(C, \Gl_n) \to M(C, \Gl_1)
\end{equation}
sends a representation or a Higgs bundle to its associated determinant representation or Higgs bundle. It is an \'{e}tale locally trivial fibration with monodromy group $\Gamma\simeq (\ZZ/2\ZZ)^{2g}$ and fibre isomorphic to $M(C, \Sl_n)$. The quotient of $M(C, \Sl_n)$ for the residual action of $\Gamma$ is $M(C, \PGl_n)$. Hence, there exist morphisms of mixed Hodge structures
\begin{align}
H^*(M(C, \Gl_n)) & \simeq H^*(M(C, \Sl_n))^{\Gamma} \otimes H^*(M(C, \Gl_1)) \label{eq:GLtoSL},\\
H^*(M(C, \PGl_n)) & \simeq H^*(M(C, \Sl_n))^{\Gamma}. \label{eq:PGLtoSL}
\end{align}
Analogous splittings hold for $IH^*$, $H_c^*$, $IH^*_c$. A proof of these facts can be found for instance in \cite[\S 3.2]{FelisettiMauri2020}. 

 The analogues of Theorems \ref{thm:IEDol} and \ref{thm:IP} for $\PGl_2$ can be obtained by substituting all the occurrence of the coefficient $2^{2g}$ with $1$ in the formulas of the theorems; see \cref{rmk:replace2gwith1}. According to \eqref{eq:GLtoSL}, the corresponding polynomials for $\Gl_2$ are the product of the E-polynomials or the (intersection) Poincar\'{e} polynomial for $\PGl_2$ with $E(\TJac)=(uv)^g(1-u)^g(1-v)^g$ or $P_t(\Jac)=(t+1)^{2g}$ respectively.

Here, as a corollary of Theorems \ref{thm:IEB}, \ref{thm:IP} and \eqref{eq:PGLtoSL}, we study the portion of $IH^*(M(C, \Sl_2))$ on which $\Gamma$ acts non-trivially. The following should be considered the untwisted analogue of \cite[Proposition 8.2]{HauselThaddeus03} in rank 2. This suggests that intersection cohomology may be the right cohomology theory to formulate a topological mirror symmetry conjecture for $M(C, \Sl_n)$ and $M(C, \PGl_n)$; see \cite[Remark 3.30]{Hausel13} and \cite{Mauri20II}.

\begin{cor}\label{thm: varinvarIE}
The variant intersection E-polynomial and Poincar\'{e} polynomials for the action of $\Gamma$ are
\begin{align*}
IE_{\mathrm{var}}(\MB(C, \Sl_2))  = & \frac{1}{2} (2^{2g}-1) q^{2g-2} ((q+1)^{2g-2}+(q-1)^{2g-2}),\\
IE_{\mathrm{var}}(\MDol(C, \Sl_2))  = & \frac{1}{2} (2^{2g}-1) (uv)^{3g-3} ((u+1)^{g-1}(v+1)^{g-1}\\
& +(u-1)^{g-1}(v-1)^{g-1}),\\
IP_{t, \mathrm{var}}(M(C, \Sl_2))  = & \frac{1}{2}(2^{2g}-1) t^{4g-4}((t+1)^{2g-2}+(t-1)^{2g-2}).
\end{align*}
In particular, $q^{-2g+2}IE_{\mathrm{var}}(\MB(C, \Sl_2))$ and $q^{-6g+6}IE_{\mathrm{var}}(\MDol(C, \Sl_2);q,q)$ are palindromic polynomials of degree $2g-2$.
\end{cor}

As a byproduct, we also obtain the E-polynomials and the Poincar\'{e} polynomial of the Kirwan--O'Grady desingularization $\pi_T\colon T \to M$. We write explicitly the E-polynomials of $M(C, \Sl_2)$, and leave the straightforward computations of the other polynomials to the reader (cf \S \ref{sec:SLtoPGL} and \S\ref{sec:semiproj}).

\begin{thm}\label{thm:EDolT} The E-polynomials of $T(C, \Sl_2)$ are
\begin{align*}
    E(  T_{\mathrm{B}}(C,  \Sl_2))  & = (q^{2g-2}+1)(q^2-1)^{2g-2}+\frac{1}{2}q^{2g-3}(q^2+1)((q+1)^{2g-2}\\ 
    & -(q-1)^{2g-2}) +2^{2g-1}q^{2g-2}((q+1)^{2g-2}+(q-1)^{2g-2})\\
    & + \frac{1}{2} q((1+q)^{2g-1}(1+q^{2g-3}) +(1-q)^{2g-1}(1-q^{2g-3}))\frac{1-q^{2g-3}}{1-q}\\
    & + 2^{2g}\frac{q}{(1-q)^3(1-q^2)} ( 2 - q - q^3  - q^{2g-4} - 2 q^{2g-2} + 
  q^{2g-1}   \\
  & - 2 q^{2g} + 4 q^{2g+1} - q^{2g+2} + q^{2g+3} + q^{4g-6} - q^{4g-5} + 4 q^{4g-4}  \\
  & - 2 q^{4g-3} + 
  q^{4g-2} - 2 q^{4g-1} - q^{4g+1} - q^{6g-6} - 
  q^{6g-4} + 2 q^{6g-3}),
  \end{align*}
  \begin{align*}
    E( \TDol(C,  \Sl_2)) & = E( \MDol(C, \Sl_2)^{\mathrm{sm}})+ \\
    & + \frac{1}{2}(uv)^g \frac{(1-(uv)^{2g-2})}{(1-uv)}\bigg( \frac{(1-u)^g(1-v)^g(1-(uv)^{2g-3})}{1-uv} \\
    & + \frac{(1+u)^g(1+v)^g(1+(uv)^{2g-3})}{1+uv}\bigg)-2^{2g} \frac{(1-(uv)^{2g-2})^2}{(1-(uv)^2)(1-uv)}\\
    & + 2^{2g}\frac{\left(1-(uv)^{2g-2}\right) \left(1-(uv)^{2g}\right)}{\left(1-uv\right)^3 
 \left(1-(uv)^2\right)}(1-(uv)^4 -(uv)^{2
 g-3}\\
 & -(uv)^{2g-1}+2(uv)^{2g}).
\end{align*}
In particular, the E-polynomial $E(T_B)$ is palindromic.
\end{thm}
\subsection{Outline and relation with other work}
\begin{itemize}
    \item In \S \ref{sec:preliminaries} we collect some preliminary results: the intersection cohomology of an affine cone, the decomposition theorem, some properties of the mixed Hodge structures of singular semi-projective varieties (cf \cite{HauselVillegas15}), and the stable isosingularity principle (implicitly used in \cite[p.834]{KiemYoo08}). Analogous degeneration techniques employed to establish the stable isosingularity principle have been discussed in \cite{deCataldoMaulikShen2019} and \cite[\S 4]{deCataldoMaulikShen2020} in relation to twisted character varieties.
    \item In \S \ref{sec:Kirwan--O'Grady desingularization} we describe the singularities of $M$ and the geometry of the Kirwan--O'Grady desingularization $\pi_T\colon T \to M$. This part highly relies on \cite{OGrady99} and \cite{KiemYoo08}. The computation of the Poincar\'{e} polynomial of the incidence variety $I_{2g-3}$ in \cite{KiemYoo08} contains a mistake, and we fix it in \S \ref{sec:incidencevariety}. 
    \item In \S \ref{sec:intersectioncohomologyoflocalmodels} we use several times the decomposition theorem to determine the intersection cohomology of the normal slice to strata of a Whitney stratification of $M$. 
    
    The singularities of the Betti and Dolbeault moduli spaces are locally modelled on Nakajima quiver varieties which usually do not admit a symplectic resolution. Although a lot is known about the intersection cohomology of quiver varieties with symplectic resolutions (see for instance \cite{Nakajima10}), the local computations in \S \ref{sec:intersectioncohomologyoflocalmodels} seem new.
    
    \item In \S \ref{sec:decompositionthm} we complete the proof of \cref{introthm:decompositiontheorempiT}. Then in \S \ref{sec:applicationdthm} and \S \ref{sec:SL2PGL2GL2} we argue how to compute the intersection E-polynomials and intersection Poincar\'{e} polynomial of $M$, and we prove the results of \S\ref{sec:computations}. The E-polynomial of $M$ are known thanks to \cite{MartinezMunoz16}, \cite{BaragliaHekmati17} and \cite{KiemYoo08}, while the ordinary Poincar\'{e} polynomial of $M$ appears in \cite{DaskalopoulosWentworth10}. Despite the active research in the field and the stimuli from the PI=WI conjecture, there are few previous works exhibiting explicit computations of the intersection cohomology of Dolbeault and Betti moduli spaces; see \cite{Felisetti2018} and \cite{FelisettiMauri2020}. 
    
    In \cite{Kiem06} Kiem studied the intersection cohomology of character varieties with coefficients in a compact Lie group. However, the methods in \cite{Kiem06} do not extend to the complex reductive case, since for a general complex reductive group $G$ the representation space $\Hom(\pi_1(C), G)$ is not smooth, and the quotient map $\Hom(\pi_1(C), G) \to \Hom(\pi_1(C), G)\sslash G$ is not placid in the sense of \cite{GM85}.
    
    We mention another remarkable precedent. The Dolbeault moduli space is a partial compactification of the cotangent bundle of the moduli space of stable vector bundles. The intersection cohomology of the moduli space of semistable vector bundles was determined in \cite{Kirwan86} for rank two (or equivalently in \cite{Kiem06}), and in \cite{MozgovoyReineke15} in full generality. It is unclear how this results may imply \cref{introthm:decompositiontheorempiT}.
    \item \S \ref{sec:P=Wconjectures} explores many implications concerning the P=W conjectures for $M$ stemmed from the previous calculations. 
\end{itemize}

\subsection{Acknowledgement}
I would like to acknowledge useful conversations and email exchange with Mark de Cataldo, Camilla Felisetti, Tamas Hausel, Daniel Huybrechts, Young-Hoon Kiem, Luca Migliorini, Hiraku Nakajima, Giulia Sacc\`{a}, Travis Schedler and Richard Wentworth. This work have been supported by the Max Planck Institute for Mathematics.

\section{Preliminaries}\label{sec:preliminaries}
\subsection{Intersection cohomology of affine cones}

Let $X$ be a complex projective variety of dimension $n-1$ with an ample line bundle $L$. The graded ring associated to $L$ is the graded $\CC$-algebra
\[R(X, L) \coloneqq \bigoplus_{m \geq 0} H^0(X, L^m).\]
The affine cone over $X$ with conormal bundle $L$ is 
\[C(X,L) \coloneqq \Spec R(X, L).\]
Let $s_1, \ldots, s_N$ be a set of generators for $R(X,L)$ of degree $m_1, \ldots, m_N$. Then there exists an embedding $C(X,L) \subseteq \CC^N$ such that $C(X,L)$ is invariant with respect to the $\Gm$-action 
\begin{equation}\label{Gmactiononcone}
  t \cdot (x_1, \ldots, x_N) = (t^{m_1}x_1, \ldots, t^{m_N}x_N).  
\end{equation}
Conversely, any affine variety with a $\Gm$-action and a fixed point which is attractive for $t \to 0$ is isomorphic to an affine cone; see for instance \cite[\S 3.5]{Demazure1988}.

All the singularities of this paper are locally modelled on affine cones, whose coordinate rings are not necessary generated in degree one. For this reason here we compute their intersection cohomology, thus generalising \cite[Example 2.2.1]{deCataldoMigliorini09}.
\begin{prop}[Intersection cohomology of an affine cone]\label{prop: interseccone}
\begin{equation*}
IH^d(C(X,L)) \simeq 
\begin{cases}
IH^d_{\mathrm{prim}}(X) & \text{ for } d< n\\
0 & \text{ for } d \geq n,
\end{cases}
\end{equation*}
where $IH^d_{\mathrm{prim}}(X) \coloneqq \ker (c_1(L) ^{n-d} \cup \colon IH^d(X) \to IH^{2n-d}(X))$ is the primitive intersection cohomology. 
\end{prop}
\begin{proof}
Denote by $C(X,L)^* \coloneqq C(X,L) \setminus \{\text{vertex}\}$ the punctured affine cone. By \cite[Lemma 1]{Durfee95} or \cite[Proposition 4.7.2]{KirwanWoolf06}, we can write
\[
IH^d(C(X,L)) \simeq 
\begin{cases}
IH^d(C(X,L)^*) & \text{ for } d< n\\
0 & \text{ for } d\geq n.
\end{cases}
\]

Suppose now that $R(X,L)$ is generated in degree one. Then the blow-up of the origin 
\[p\colon BC(X,L) \coloneqq \Spec_X \bigoplus_{m \geq 0} L^m \to C(X,L)\]
is the total space of the line bundle $L^*$. By the hard Lefschetz theorem, the relative long exact sequence of the inclusion $C(X, L)^* \hookrightarrow BC(X,L)$ splits into the short exact sequences
\begin{equation}\label{ses:thomforIC}
  0 \to IH^{d-2}(X) \xrightarrow{c_1(L) \cup} IH^{d}(X) \to IH^{d}(C(X, L)^*) \to 0 \quad \text{ for } d < n.  
\end{equation} 
Therefore, we obtain that for $d < n$
\[IH^{d}(C(X, L)) \simeq \mathrm{coker} (c_1(L) \cup\colon IH^{d-2}(X) \to IH^{d}(X)) \simeq  IH^d_{\mathrm{prim}}(X).\]

If $R(X,L)$ is not generated in degree one, then $BC(X,L)$ is the total space of a line bundle only up to a finite cover; see \cite[\S 1.2]{OrlikWagreich71}. More precisely, consider the finite morphism $g\colon \CC^N \to \CC^N$ defined by
\[g(x_1, \ldots, x_N) = (x^{m_1}_1, \ldots, x^{m_N}_N).\] Set $V'= g^{-1}C(X,L)$, where $C(X,L)$ is embedded in $\CC^N$ as in \eqref{Gmactiononcone}.
We see that $C(X,L)$ is the quotient of $V'$ by the finite group $A= (\Z/m_1 \Z) \times \ldots \times (\Z/m_N \Z)$ acting on $V'$ by coordinatewise multiplication. 

$V'$ has a $\Gm$-action defined by $t \cdot (x_1, \ldots, x_N) = (t x_1, \ldots, t x_N)$, and covering the $\Gm$-action on $C(X,L)$ given by \eqref{Gmactiononcone}. Since the $\Gm$-action on $V'$ has weight one, $X'$ is the spectrum of a graded algebra generated in degree one, say $V' = C(X',L')$ for some projective variety $X'$ and ample line bundle $L'$. In particular, there exists a commutative diagram
\begin{equation}\label{blowupcone}
\begin{tikzpicture}[baseline=(current  bounding  box.east)]
\node (A) at (0,1.2) {$C(X',L')$};
\node (B) at (3,1.2) {$BC(X',L')$};
\node (E) at (6, 1.2) {$X'$};
\node (C) at (0,0) {$C(X,L)$};
\node (D) at (3,0) {$BC(X,L)$};
\node (F) at (6, 0) {$X$};
\draw[->,thick] (A) -- (C) node [midway,left] {$g$};
\draw[->,thick] (B) -- (A) node [midway,above] {$p'$};
\draw[->,thick] (B) -- (D) node [midway,right] {};
\draw[->,thick] (D) -- (C) node [midway,above] {$p$};
\draw[->,thick] (E) -- (B) node [midway,above] {$i'$};
\draw[->,thick] (F) -- (D) node [midway,above] {$i$};
\draw[->,thick] (E) -- (F) node [midway,right] {};
\end{tikzpicture}
\end{equation}
where $p$ and $p'$ are blow-ups of the vertices of the cones, $i$ and $i'$ are the embedding of the exceptional divisors, and the vertical arrows are quotients with respect to (the lift of) the action of $A$. Thus we have $IH^*(C(X',L'))^A \simeq IH^*(C(X, L))$.

The discussion above shows that the sequences \eqref{ses:thomforIC} are exact for $C(X',L')$, and it is $A$-equivariant by the commutativity of \eqref{blowupcone}. Taking invariants, we show then that \eqref{ses:thomforIC} holds for $C(X,L)$ unconditionally.  
\end{proof}
\subsection{The perverse Leray filtration}
In this section we briefly recall the statement of the decomposition theorem and the definition of the perverse filtration. 

For a complex algebraic variety $X$ let $D^b(X, \QQ)$ be the bounded derived category of complexes of sheaves of $\QQ$-vector spaces with algebraically constructible cohomology. Denote  the full abelian subcategory of perverse sheaves by $\Perv(X)$ and the perverse cohomology functors by ${}^{\mathfrak{p}}\mathcal{H}^i\colon D^b(S, \QQ) \to \Perv(X)$; see \cite{BeilinsonBernsteinDeligne1981} or \cite{deCataldoMigliorini09}.

Let $MHM_{\mathrm{alg}}(X)$ be the category of algebraic mixed Hodge modules with rational coefficients, and $D^bMHM_{\mathrm{alg}}(X)$ its bounded derived category. Let ${}^{\mathfrak{p}}\mathcal{H}^i\colon \newline D^bMHM_{\mathrm{alg}}(X) \to MHM_{\mathrm{alg}}(X)$ be the cohomology functors; see \cite{Saito90} or \cite{Schnell14}.

The simple objects of $D^bMHM_{\mathrm{alg}}(X)$ (resp. $D^b(X, \QQ)$) are the intersection cohomology complexes $IC_X(L)$, where $L$ is a polarizable variation of pure
Hodge structures (resp. a local system) on a Zariski-open subset of the smooth locus of $X$. We denote simply by $IC_X$ the complex $IC_X(\QQ_{X \setminus \Sing(X)})$. In particular, $IH^d(X) \simeq H^{d}(X, IC_X(\QQ_{X \setminus \Sing(X)})[-\dim X])$.

There is a forgetful functor $\rat\colon D^bMHM_{\mathrm{alg}}(X) \to D^b(X, \QQ)$ which commutes with ${}^{\mathfrak{p}}\mathcal{H}^i$ and pushforward $Rf_*$, and maps $MHM_{\mathrm{alg}}(X)$ in $\Perv(X)$. We will make no notational distinction between $K \in D^bMHM_{\mathrm{alg}}(X)$ and $\rat(K)$.

Now let $f\colon X \to Y$ be a proper morphism of varieties with defect of semismallness
\[
r \coloneqq \dim X \times_Y X - \dim X.
\]
The decomposition theorem of Beilinson--Bernstein--Deligne--Gabber, or its mixed Hodge module version by Saito, says that there is an isomorphism in $D^bMHM_{\mathrm{alg}}(X)$ (resp. in $D^b(X, \QQ)$)
\[
Rf_*IC_X = \bigoplus^r_{i=-r} {}^{\mathfrak{p}}\mathcal{H}^i(Rf_*IC_X)[-i] = \bigoplus^r_{i=-r} \bigoplus_l IC_{\overline{Y}_{i,l}}(L_{i,l})[-i],
\]
where $L_{i, l}$ are polarizable variations of pure
Hodge structures (resp. local systems) on the strata of a stratification $Y = \bigsqcup_{l} Y_{i,l}$; see \cite{BeilinsonBernsteinDeligne1981} and \cite{Saito89}.

The perverse (Leray) filtration is
\[
P_k IH^d(X)= \mathrm{Im} \bigg\{ H^d(\bigoplus^{k-r}_{i=-r} {}^{\mathfrak{p}}\mathcal{H}^i(Rf_*IC_X)[-i-\dim X])\to IH^d(X)\bigg\}.
\]
When $Y$ is affine, de Cataldo and Migliorini provided a simple geometric characterization of the perverse filtration; see \cite[Theorem 4.1.1]{deCataldoMigliorini2010}. Let $\Lambda^s\subset Y$ be a general $s$-dimensional affine section of $Y\subset \Aff^N$. Then 
\begin{equation}\label{prop:characperverse}
    P_k IH^d(X) = \mathrm{Ker}\left\lbrace IH^d(X)\rightarrow IH^d(f^{-1}(\Lambda^{d-k-1}))
\right \rbrace.
\end{equation}
This means that the cocycle $\eta \in IH^d(X)$ belongs to $P_kIH^d(X)$ if and only if its restriction to $f^{-1}(\Lambda^{d-k-1})$ vanishes, i.e.\ $\eta|_{f^{-1}(\Lambda^{d-k-1})}=0$. 

\subsection{Mixed Hodge structure of semi-projective varieties}\label{sec:semiproj}
In order to compute the intersection Poincar\'{e} polynomial of $M$, we observe that $\MDol$ and $\TDol$ are semi-projective.

\begin{defn} \emph{\cite[Definition 1.1.1]{HauselVillegas15}}
A \textbf{semiprojective} variety is a complex quasi-projective variety $X$ with a $\Gm$-action such that:
\begin{itemize}
\item the fixed point set $\Fix(X)$ is proper;
\item for every $m \in X$ the limit $\lim_{\lambda \to 0} \lambda \cdot m$ exists.
\end{itemize}
The \textbf{core} of $X$, denoted $\Core(X)$, is the (proper) union of the repelling sets of $\Fix(X)$; see \emph{\cite[Corollary 1.2.2]{HauselVillegas15}}.
\end{defn}
\begin{prop}\label{prop:retracttothecore}
Let $X$ be a semi-projective variety. Then the inclusion $\Core(X) \subset X$ is a homotopy equivalence.
\end{prop}
\begin{proof}
The flow $\RR^+ \times X \to X$, induced by the $\Gm$-action, defines a deformation retract of $X$ onto a neighbourhood of $\Core(X)$, which in turn is homotopy equivalent to $\Core(X)$.
\end{proof}

\begin{prop}\label{prop:purityandother}
If $X$ is a semiprojective variety, then 
\begin{enumerate}
    \item the mixed Hodge structure on $IH^*(X)$ is pure;
    \item $W_{d-1}H^d(X) = \ker \{ H^{d}(X) \to IH^{d}(X)\}$, and $W_d H^d(X) = H^d(X)$.
\end{enumerate}
\end{prop}
\begin{proof}
Let $f\colon\widetilde{X} \to X$ be a $\Gm$-equivariant resolution of singularities of $X$. Then $\widetilde{X}$ is smooth and semiprojective, and it has pure cohomology by \cite[Corollary 1.3.2]{HauselVillegas15}. Via the decomposition theorem, the mixed Hodge structure on $IH^*(X) \subset H^*(\tilde{X})$ is pure too. 

As $X$ retracts onto the proper algebraic variety $\Core(X)$, the weight filtration on $H^*(X)\simeq H^*(\Core(X))$ is concentrated in degree $[0, 2d]$ by \cite[Theorem 5.39]{PetersSteenbrink2008},  i.e.\ $W_d H^d(X) = H^d(X)$. 

The resolution $f$ induces a surjective morphism $\Core(\widetilde{X}) \to \Core(X)$ of proper algebraic varieties. Hence, by \cite[Corollary 5.43]{PetersSteenbrink2008} we have
\[W_{d-1}H^d(\Core(X)) = \ker  \{f^*\colon H^d(\Core(X)) \to H^d(\Core(\tilde{X}))\},\]
and so  $W_{d-1}H^d(X) = \ker \{ f^*\colon H^{d}(X) \to H^d(\widetilde{X})\}$ by \cref{prop:retracttothecore}. Finally, since $f^*$ factors as $H^d(X) \to IH^d(X) \hookrightarrow H^d(\widetilde{X})$, we conclude that \[W_{d-1}H^d(X) = \ker \{ H^{d}(X) \to IH^{d}(X)\}.\]
\end{proof}

The multiplicative group $\Gm$ acts on $\MDol(C, \Sl_n)$ by $\lambda \cdot (E, \phi)=(E, \lambda \phi)$.
The \textbf{Hitchin fibration} 
\begin{equation}\label{Hitchinfibration}
    \chi\colon  \MDolSl \to \bigoplus^{n}_{i=2} H^0(C, K_C^{\otimes i})
\end{equation}
assigns to $(E, \phi)$ the characteristic polynomial of the Higgs field $\phi$. By \cite[Theorem 6.11]{Simpson1994} the map $\chi$ is a proper $\Gm$-equivariant map, where $\Gm$ acts linearly on $H^0(X, K_X^{\otimes i})$ with weight $i$. In particular, $\Fix(\MDolSl)$ is contained in the nilpotent cone $\chi^{-1}(0)$. Therefore, $\MDolSl$ is semiprojective.

The same argument works for $G=\Gl_n, \PGl_n$ as well.

\subsection{Stable isosingularity principle}\label{sec:stableisosingularityprinciple}

Let $S$ be a smooth projective K3 surface or an abelian surface. In this section we establish a stable isosingularity principle for the (non-proper) Dolbeault moduli spaces $M(C, \Gl_n)$ and $M(C, \Sl_n)$ and the (proper) Mukai moduli spaces $\Mukai$ and $\Kumm$. This means that these moduli spaces have the same analytic singularities, up to multiplication by a polydisk. The upshot is that:
\begin{itemize}
    \item the description of the local model of the singularities of $M(S,v)$ in \cite{OGrady99} or \cite{ChoyKiem07} holds for $M(C, \Gl_2)$ and $M(C, \Sl_2)$ mutatis mutandis;
    \item the same sequence of blow-ups which desingularises $M(S,v)$ in \cite{OGrady99} resolves the singularities of $M(C, \Gl_2)$ and $M(C, \Sl_2)$ mutatis mutandis;
    \item the description of the summands of the decomposition theorem in \cref{introthm:decompositiontheorempiT} holds for $M(C, \Gl_2)$, $M(C, \Sl_2)$, $\Mukai$ and $\Kumm$ with Mukai vector $v = 2 w \in H^*_{\mathrm{alg}}(S, \ZZ)$, where $w$ is primitive and $w^2=2(g-1)$.
\end{itemize}

 We briefly recall the definition of Mukai moduli space. Fix an effective Mukai vector\footnote{i.e.\ there exists a coherent sheaf $\mathcal{F}$ on $S$ such that $v = (rk(\mathcal{F}), c_1(\mathcal{F}), \chi(\mathcal{F})-\epsilon(S)rk(\mathcal{F}))$, with $\epsilon(S)\coloneqq 1$ if $S$ is K3, and  $0$ if $S$ is abelian.} $v \in H^{*}_{\mathrm{alg}}(S, \ZZ)$. 
  Define $\Mukai$ the moduli space of Gieseker $H$-semistable sheaves on $S$ with Mukai vector $v$ 
for a sufficiently general polarization $H$ (which we will typically omit in the notation); see \cite[\S 1]{Simpson1994I}.  Further, if $S$ is an abelian variety with dual $\hat{S}$, and $\dim M(S,v)\geq 6$, 
then the Albanese morphism 
$\alb\colon \Mukai \to S \times \hat{S}$
is isotrivial, and we set
$\Kumm \coloneqq \alb^{-1}(0_S, \mathcal{O}_S)$.

\begin{rmk}[Donagi--Ein--Lazarsfeld degeneration]
Mukai moduli spaces should be thought as locally trivial deformations of Dolbeault moduli spaces as follows. Fix an ample curve $C \subset S$ of genus $g \geq 2$. Donagi, Ein and Lazarsfeld shows in \cite{DEL97} that there exists a flat family 
$\pi_{\mathcal{W}}\colon \mathcal{W} \to \PP^1$ such that
\begin{enumerate}
    \item $\pi_{\mathcal{W}}^{-1}(\PP^1 \setminus \{0\}) \simeq M(S,(0, nC, -n C^2/2) \times \mathbb{A}^1$;
    \item $\pi_{\mathcal{W}}^{-1}(0)\simeq \MDol(C, \Gl_n)$.
\end{enumerate}
Unless $g=2$ and $n=2$, $M(S,(0, nC, -n C^2/2))$ and $\MDol(C, \Gl_n)$ have $\QQ$-factorial terminal symplectic singularities; see \cite[Theorem B]{KaledinLehnSorger2006} and \cite[Theorem 1.2]{BellamySchedler2019}, together with Simpson's isosingularity principle \cite[Theorem 10.6]{Simpson1994}. Hence, by \cite[Theorem 17]{Namikawa08} the morphism $\pi_{\mathcal{W}}$ is locally analytically trivial\footnote{The local triviality of $\pi_{\mathcal{W}}$ holds for $g=2$ and $n=2$ too by \cite[Proposition 2.16]{PeregoRapagnetta2013}.}. Roughly this means that the two moduli spaces have the same singularities. We make this statement precise in \cref{prop:isosingularity}.
\end{rmk}

We start by stating the notion of stable isosingularity. 

\begin{defn}\label{defn:equisingularity}
The varieties $X$ and $Y$ are \textbf{stably isosingular} if there exist complex Whitney stratifications\footnote{See \cite[Chapter 1]{GoreskyMacPherson88} for a definition.} by (smooth non-necessarily connected) locally Zariski-closed subsets $X_i$ and $Y_i$ 
such that 
\begin{enumerate}
\item $X = \bigsqcup_i X_i$ and $Y = \bigsqcup_i Y_i$; 
\item \label{item:posets} the posets of closed subsets $\{\overline{X}_i\}$ and $\{\overline{Y}_i\}$ ordered by inclusion are equal;
\item \label{equising} the normal slices through $X_i$ and $Y_i$ are locally analytically isomorphic.
\end{enumerate} 
If $\dim X= \dim Y$, then we say that $X$ and $Y$ are \textbf{isosingular}.
\end{defn}
It is implicit in \emph{\ref{equising}.} that the stratifications above are \textbf{analytically equisingular} along each stratum, i.e.\ the analytic type of the normal slices through $x \in X_i$ (resp. $y \in Y_i$) is independent of $x$ (resp. $y$). Not all algebraic variety admits such a stratification; see \cite[Example 13.1]{Whitney65}. However, the moduli spaces considered below will satisfy the following stronger condition of analytic normal triviality.
\begin{defn}
A Whitney stratification $X = \bigsqcup_i X_i$ is \textbf{analytically trivial in the normal direction} to each strata, if for any $x \in X_i$ there exists a normal slice $N_x$ through $X_i$ at $x$, and a neighbourhood of $x$ in $X$ which is locally analytically isomorphic to $N_x \times T_x X_i$ at $(x, 0)$.
\end{defn}
Note that if $X$ and $Y$ are stably isosingular via Whitney stratifications which are analytically trivial in the normal direction, and a sequence of blow-ups along (the strict transforms of) some $X_i$ gives a desingularization of $X$, then the same sequence of blow-ups along the corresponding strata $Y_i$ gives a desingularization of $Y$. In addition, if $X$ and $Y$ are isosingular, then an analytic neighbourhood of any point of $X$ is isomorphic to an analytic neighbourhood of some point in $Y$. 

\begin{exa}[Analytically trivial fibrations]\label{ex:analtrivfibration}
Let $f\colon X \to Y$ be an analytic locally trivial fibration, and suppose that $F \coloneqq f^{-1}(y)$, with $y \in Y$, admits an analytically equisingular Whitney stratification. Then $X$ and $F$ are stably isosingular. Indeed, by the local triviality, any Whitney stratification of $f^{-1}(y)$ can be lifted to a Whitney stratification on $X$ with the same normal slices. In particular, if $W$ is a smooth algebraic variety, $F$ and $F \times W$ are stably isosingular.
\end{exa}

\begin{lem}[Quadraticity of deformation spaces]\label{lem:Quadraticityofdefomrationspaces}
Let $[F] \in \MB(C, \Gl_n)$ or $M(S,v)$ be a singular point corresponding to the polystable representation or polystable sheaf $F$. Then the representation space $\Hom(\pi_1(C), \Gl_n)$ at $F$ or the deformation space $\Def_F$ (cf \emph{\cite[\S 2.A.6]{HuybrechtsLehn2010}}) is quadratic, i.e.\ it is locally isomorphic to a (reduced) complete intersection of homogeneous quadrics.
\end{lem}
\begin{proof}
This follows from the Goldaman-Millson theory \cite{GoldmanMillson1988} if $[F] \in \MB(C, \Gl_n)$, or \cite[Theorem 1.2]{ArbarelloSacca2018} if $[F] \in M(S,v)$, with $S$ K3 surface. Looking into the proof of \cite[Theorem 3.7 and 3.8]{ArbarelloSacca2018} and \cite{Zhangformal2012}, one can see that the same proof holds for $[F] \in M(S,v)$ with $S$ abelian surface.
\end{proof}

\begin{prop}\label{prop:isosingularity}
Let $S$ be a K3 or an abelian surface with $\Pic(S)\simeq \Z$ generated by the class of a curve $C$ of genus $g \geq 2$. Then $\MB(C, \Gl_n)$ and $M(S,(0, nC, -nC^2/2))$ are isosingular.
\end{prop}
\begin{proof}
Given a (quasi-projective) variety $X$ equipped with the action of a reductive group $G$, let $\xi\colon X \to Y \coloneqq X \sslash G$ be the quotient map. Any fibre $\xi^{-1}(y)$, with $y \in Y$, contains a closed $G$-orbit $T(y)$. Denote the conjugacy class of a closed subgroup $H$ of $G$ by $(H)$. Then $Y_{(H)}$ is the set of points $y \in Y$ such that the stabiliser of $x \in T(y)$ is in $(H)$. The loci $Y_{(H)}$ are the strata of the \textbf{stratification by orbit type} of $Y$.

If $Y$ is a Nakajima quiver variety\footnote{We use Nakajima quiver varieties only tangentially in this place. For brevity we omit the definition, and we refer the reader for instance to \cite{Ginzburg2012}.}, then the stratification by orbit type is a complex Whitney stratification, which is analytically trivial in the normal direction to each stratum, due to \cite[Proposition 4.2]{Mayrand2018}.

$\MB(C, \Gl_n)$ and $M(S,(0, nC, -nC^2/2))$ are $\PGl_N$-quotients, and the quadraticity of the deformation spaces imply that they are locally modelled on Nakajima quiver varieties; see \cite[Theorem 2.5]{BellamySchedler2019} and \cite[Proposition 6.1]{ArbarelloSacca2018}. 
By construction, the stratifications by orbit type of $\MB(C, \Gl_n)$ and $M(S,(0, nC, -nC^2/2))$ are locally isomorphic to stratification by orbit type of quiver varieties, and so they are complex Whitney stratifications, analytically trivial in the normal direction to each stratum.

A singular point of either moduli space is a polystable objects
\[F = F^{l_1}_1\oplus \ldots \oplus F^{l_s}_s,\]
where $F_i$ are distinct stable factors. The automorphism group of $F$ is  
\[\prod^s_{i=1} \Gl_{l_i} \subset \Gl_n,\]
which can be identified up to constants with the stabiliser of a point in $T(F)$ under the $\PGl_N$-action; see for instance \cite[\S 2.5]{KaledinLehnSorger2006}\footnote{Note that this is the only place where we use the assumption that $S$ has Picard number one. Otherwise, if $[C]$ can be decomposed in the sum of effective classes, then the rank of the automorphism group of $F\in M(S,(0, nC, -nC^2/2))$ may be greater than $n$, and the stratification by orbit type of $M(S,(0, nC, -nC^2/2))$ would have more strata than that of $\MB(C, \Gl_n)$. Observe however that a general result is achieved in \cref{Stableisosingularityprinciple}.}. 

The poset of inclusions of the orbit type strata for both the Dolbeault and Mukai moduli spaces is isomorphic to the poset of inclusion of the stabilisers of $T(F)$, and the analytic type of the normal slice through an orbit type strata is prescribed by the (abstract) isomorphism class of the stabiliser; see again \cite[Theorem 2.5]{BellamySchedler2019} and \cite[\S 2.7]{KaledinLehnSorger2006}. This gives \emph{\ref{item:posets}}. and \emph{\ref{equising}}. of \cref{defn:equisingularity}. 
The isosingularity follows from
\[\dim \MB(C, \Gl_n)=2(g-1)n^2+2=v^2+2= \dim M(S,(0, nC, -nC^2/2)).\]
\end{proof}
There exists a clear geometric argument for \cref{prop:isosingularity}, sketched below. 
\begin{proof}[Sketch of the proof of \cref{prop:isosingularity} via a degeneration argument]
Via the \\Donagi--Ein--Lazarsfeld degeneration one can actually prove that $\MDol(C, \Gl_n)$ is isosingular to a neighbourhood of a nilpotent cone of  $M(S,(0, nC, -nC^2/2))$ as defined in \cite[\S 2]{DEL97}. This is an analytic open set of the Mukai moduli space that intersects all the orbit type strata, if $C$ generates the Picard group of $S$. In order to extend the result to the whole Mukai moduli space, it is sufficient to invoke the analytic triviality in the normal direction of the stratification by orbit type of $M(S,(0, nC, -nC^2/2))$, which follows from the quadraticity of the deformation spaces \cref{lem:Quadraticityofdefomrationspaces}.
\end{proof}

\begin{thm}[Stable isosingularity principle] \label{Stableisosingularityprinciple} Let $C$ be a curve of genus $g \geq 2$, and let $S$ be a K3 or an abelian surface. Fix a Mukai vector $v = n w \in H^*_{\mathrm{alg}}(S, \ZZ)$, where $w$ is primitive and $w^2=2(g-1)$. 

Then $M(C, \Gl_n)$, $M(C, \Sl_n)$, $\Mukai$ and $\Kumm$ are stably isosingular.
\end{thm}
\begin{proof}
\begin{itemize}
     \item $\MDolG$ and $\MBG$ are isosingular by \cite[Theorem 10.6]{Simpson1994}, independently on the complex structure of $C$.
     \item Now let $S'$ be a K3 or an abelian surface such that $C$ embeds in $S'$ and generates its Picard group. Then $M(S',(0, nC, -nC^2/2))$ and $M(S,v)$ (resp. $K(S',(0, nC, -nC^2/2))$ and $K(S,v)$) are isosingular by \cite[Theorem 1.17]{PeregoRapagnetta18}, independently on the complex structure of $S$.
    \item $\MB(C, \Gl_n)$ and $M(S',(0, nC, -nC^2/2))$ are isosingular by \cref{prop:isosingularity}.
    \item Let $S$ be an abelian surface. The morphisms $\alb\colon \Mukai \to S \times \hat{S}$ and $\alb\colon \MDol(C, \Gl_n) \to M(C, \Gl_1)$, given by $\alb((E, \phi))=(\det E, \text{tr} \phi)$, are \'{e}tale locally  trivial fibrations with fibers $\MDol(C, \Sl_n)$ and $\Kumm$ respectively. The restriction of $\alb$ to the orbit type strata are \'{e}tale locally trivial too. This means that there exists a neighbourhood of $[F] \in M(S,v)_{(H)}$ locally analytically isomorphic to 
   \begin{equation}\label{eq:etalelocalorbittype}
       N_{[F]} \times T_{[F]}M(S,v)_{(H)} \simeq N_{[F]} \times T_{[F]}\Kumm_{(H)} \times T_{\alb([F])}(S \times \hat{S})
   \end{equation}
    at $([F], 0)$, where $N_{[F]}$ is a normal slice through  $\Kumm_{(H)}$ at $[F]$. Further, the morphism $\alb$ is locally given by the linear projection onto the last factor of \eqref{eq:etalelocalorbittype} by \cref{lem:Quadraticityofdefomrationspaces}. 
    The same argument works for $\MDol(C, \Sl_n)$ too. As in \cref{ex:analtrivfibration}, we conclude that $\MDol(C, \Gl_n)$ and $\MDol(C, \Sl_n)$ (resp. $\Mukai$ and $\Kumm$) are stably isosingular.
\end{itemize}
\end{proof}
\section{Kirwan--O'Grady desingularization}\label{sec:Kirwan--O'Grady desingularization}
\subsection{Singularities of $M$}
Recall that $M$ denotes indifferently the moduli spaces $\MBG$ or $\MDolG$ with $G=\Gl_2$ or $\Sl_2$. The stratification by orbit type of $M$ (cf \S \ref{sec:stableisosingularityprinciple}) determines a filtration by closed subsets
\[M \supset \Sigma \coloneqq \Sing M \supset \Omega \coloneqq \Sing \Sigma.\]
In this section we characterise $\Sigma$, $\Omega$ and their normal slices, maily appealing to \cite{OGrady99}. 

\begin{prop}\label{prop:sing generalities}
\begin{enumerate}
    \item $M(C, G)$ is an algebraic variety of dimension $6g-6$, if $G=\Sl_2$, or $8g-6$ if $G=\Gl_2$. 
    \item \label{item:singularlocusstrictlysemistable}The singular locus of $M$ is the subvariety of strictly semi-simple Higgs bundles or representations.
    \item If $g \geq 3$, $M$ is factorial with terminal symplectic singularities. If $g=2$, $M$ admits a symplectic resolution. 
\end{enumerate} 
\end{prop}
\begin{proof}
The statements have been proved for $\MB(C, \Gl_2)$ in \cite[Theorem 1.1, 1.2, 1.5, Lemma 2.8]{BellamySchedler2019}. The same holds for $M$ by the stable isosingularity principle (\cref{Stableisosingularityprinciple}), possibly with the exception of the factoriality. However, to show that $M$ is factorial, one can repeat the argument of \cite[Theorem 1.2]{BellamySchedler2019} words by words.
\end{proof}

\begin{prop}[Singularities of $M$]\label{prop:singulMB} 
\begin{enumerate}
     \item The singular locus of $\MB (C, \Sl_2)$, denoted $\Sigma_{\mathrm{B}} (C, \Sl_2)$, is  
     \[\left\lbrace (A_1, B_1, \ldots, A_g, B_g) \in (\CC^*)^{2g} \subset \Sl^{2g}_2\right \rbrace\sslash \Sl_2 \simeq (\CC^*)^{2g}/(\ZZ/2\ZZ), \]
 where $\ZZ/2\ZZ$ acts on $(\CC^*)^{2g}$ by $v \mapsto -v$. Set $\Sigma_{\iota, \mathrm{B}} (C, \Sl_2) \coloneqq (\CC^*)^{2g}$.
 \item The singular locus of $\MB (C, \Gl_2)$, denoted $\Sigma_{\mathrm{B}} (C, \Gl_2)$, is  
     \[\left\lbrace (A_1, B_1, \ldots, A_g, B_g) \in (\CC^*)^{2g} \times (\CC^*)^{2g} \subset \Gl^{2g}_2\right \rbrace\sslash \Gl_2,\] 
     which is isomorphic to the second symmetric product of $(\CC^*)^{2g}$. \newline Set $\Sigma_{\iota, B} (C, \Gl_2) \coloneqq (\CC^*)^{2g} \times (\CC^*)^{2g}$.
     \item The singular locus of $\MDol (C, \Sl_2)$, denoted $\Sigma_{\mathrm{Dol}} (C, \Sl_2)$, is  
     \[\lbrace (E, \Phi) | \, (E, \Phi)  \simeq (L, \phi) \oplus (L^{-1}, -\phi), \, L \in \operatorname{Jac}(C), \, \phi \in H^0(C, K_C)  \rbrace,\]
     which is isomorphic to \[(\operatorname{Jac}(C) \times H^0(C, K_C))/(\ZZ/2\ZZ) \simeq \TJac/(\ZZ/2\ZZ),\] where $\ZZ/2\ZZ$ acts on $\operatorname{Jac}(C)$ by $L \mapsto L^{-1}$ and on $H^0(C, K_C)$ by $\phi \mapsto - \phi$. Set $\Sigma_{\iota, Dol} (C, \Sl_2) \coloneqq \TJac$.
     \item The singular locus of $\MDol (C, \Gl_2)$, denoted $\Sigma_{\mathrm{Dol}} (C, \Gl_2)$, is  
     \[\lbrace (E, \Phi) | \, (E, \Phi)  \simeq (L, \phi) \oplus (L', \phi'), \, L, L' \in \operatorname{Jac}(C), \, \phi, \phi' \in H^0(C, K_C)  \rbrace,\]
     which is isomorphic to the second symmetric product of $\TJac$. Set $\Sigma_{\iota, Dol} (C, \Gl_2) \coloneqq \TJac \times \TJac$.
     \item The singular locus of $\Sigma (C, \Sl_2)$, denoted $\Omega (C, \Sl_2)$, is a set of $2^{2g}$ points.
     \item The singular locus of $\Sigma (C, \Gl_2)$, denoted $\Omega (C, \Gl_2)$, is isomorphic to $M(C, \Gl_1)$.
 \end{enumerate}
 \end{prop}
 \begin{proof}
The results follows easily from \cref{prop:sing generalities}.\eqref{item:singularlocusstrictlysemistable}.
 \end{proof}

 \begin{defn}\label{def:involution}
 \begin{enumerate}
    \item By \cref{prop:singulMB}, there exists a double cover $q\colon \Sigma_{\iota} \to \Sigma$ branched along $\Omega$.
     \item The involution $\iota\colon \Sigma_{\iota} \to \Sigma_{\iota}$ is the deck transformation of $q$.
     \item The largest open subset of $\Sigma_{\iota}$ where $q$ is \'{e}tale is denoted $\Sigma^{\circ}_{\iota} \coloneqq q^{-1}(\Sigma \setminus \Omega)$. 
     \item There exists a rank-one local system $\mathscr{L}$ on $\Sigma^{\circ} \coloneqq \Sigma \setminus \Omega$ such that
 \[q_* \mathbb{Q}_{\Sigma^{\circ}_{\iota}} = \mathbb{Q}_{\Sigma^{\circ}} \oplus \mathscr{L}.\]
 \end{enumerate} 
 \end{defn}
 \begin{prop}[Normal slices]\label{prop:normalslices}
 \begin{enumerate}
 \item\label{item:localmodelSigma} A slice $N_{\Sigma}$ normal to $\Sigma$ at a point in $\Sigma \setminus \Omega$ is locally analytically isomorphic to an affine cone over the incidence variety
 \[I_{2g-3} \coloneqq \bigg\lbrace ([x_i], [y_j]) \in \PP^{2g-3} \times \PP^{2g-3} \, \big| \, \sum^{2g-3}_{k=0} x_k y_k =0\bigg\rbrace\]
 with conormal bundle $\mathcal{O}(1,1) \coloneqq (\mathcal{O}_{\PP^{2g-3}}(1) \boxtimes \mathcal{O}_{\PP^{2g-3}}(1))|_{I_{2g-3}}$. 
 \item \label{item:localmodelOmega}  Let $(W, q)$ be a vector space of dimension $3$ endowed with a quadratic form $q$ of maximal rank, and $(V, \omega)$ be a symplectic vector space of dimension $2g$. Let $\Hom^{\omega}(W, V)$ be the cone of linear maps from $W$ to $V$ whose image is isotropic. Note that the group $SO(W)$ acts on $\Hom^{\omega}(W, V)$ by pre-composition. 
 
 Then a normal slice $N_{\Omega}$ through $\Omega$ is isomorphic to an affine cone over $\PP\Hom^{\omega}(W, V) \sslash SO(W)$.
\end{enumerate}
\end{prop}
\begin{proof}
The local models have been described in \cite[(3.3.2)]{OGrady99} (see also \cite[Proposition 3.2.(2)]{ChoyKiem07}) and in \cite[(1.5.1)]{OGrady99} (together with \cref{lem:Quadraticityofdefomrationspaces}) for $M(S,v)$ with $v=(2,0,-2c)$. 
The description holds for $M$ too by the stable isosingularity principle (\cref{Stableisosingularityprinciple}).
\end{proof}

\subsection{Geometry of the desingularization}\label{sec:desingularization}
Inspired by \cite{Kirwan85}, O'Grady exhibits a desingularization of the Mukai moduli spaces $M(S,v)$ of semistable sheaves on a projective K3 surface $S$ with Mukai vector $v=(2,0,-2c) \in H^*_{\mathrm{alg}}(S, \ZZ)$. By the stable isosingularity principle (cf \S \ref{sec:stableisosingularityprinciple}) the same sequence of blow-ups gives a desingularization of $M$. In this section, we recall the geometry of the exceptional locus, and we compute the E-polynomials of its strata.
\begin{prop}[Kirwan--O'Grady desingularization]
Let 
\begin{itemize}
    \item $\pi_R \colon R \to M$ be the blow-up of $M$ along $\Omega$;
    \item $\pi_S \colon S \to R$ be the blow-up of $R$ along $\Sigma_R \coloneqq \pi^{-1}_{R, *} \Sigma$;
    \item $\pi_T \colon T \to S$ be the blow-up of $S$ along its singular locus.
\end{itemize}
Then the composition $\pi \coloneqq \pi_T \circ \pi_S \circ \pi_R\colon T \to M$ is a log resolution of $M$\footnote{In genus 2 the (unique) symplectic resolution of $M$ can be obtained by contracting a $\PP^2$-bundle in $S$. See \cite[Proposition 8.6]{FelisettiMauri2020} where $\pi_R$ and $\pi_S$ are denoted $\eta$ and $\zeta$ respectively. In particular, the resolution $\pi$ is not symplectic.}.
\end{prop}
\begin{proof}
It follows from \cite[1.8.3]{OGrady99} and \cref{Stableisosingularityprinciple}.
\end{proof}

\begin{notation}
\begin{itemize}
    \item $D_1$, $D_2$ and $D_3$ are (the strict transform of) the exceptional divisors in $T$ of the blow-ups $\pi_R$, $\pi_S$ and $\pi_T$ respectively. 
    \item $D_{ij} \coloneqq D_i \cap D_j$ and $D_{123} \coloneqq D_1 \cap D_2 \cap D_3$ are (smooth closed) strata of the exceptional locus of $\pi$. 
\item $I_{k} \coloneqq \{ ([x_0: \ldots: x_k],[y_0:\ldots:y_k])\in \PP^k \times \PP^k | \sum^k_{i=0} x_iy_i=0\}$.
    \item $\Hom^{\omega}_k(W, V)$ is the subspace of linear maps in $\Hom^{\omega}(W, V)$ of rank $\leq k$.
    \item $\Gr^{\omega}(k,V)$ is the Grassmanian of $k$-dimensional linear subspaces of $V$, isotropic with respect to the symplectic form $\omega$.
    \item $\hat{\PP}^5$ is the blow-up of $\PP^5 \simeq \PP(S^2(W))$ (space of quadratic forms on $W$) along $\PP^2$ (locus of quadratic form of rank $1$).
    \item $\hat{Q}$ is the blow-up of $Q \subset \PP(S^2(W))$ (space of degenerate quadratic forms on $W$) along $\PP^2$ (locus of quadratic form of rank $1$).
\end{itemize}
\end{notation}
\begin{prop}[Geometry of the blow-ups $\pi_R$, $\pi_S$ and $\pi_T$]\label{prop:geometryblowup} \begin{enumerate} 
\item\label{item:exclocusoverOmega} The preimages $\pi^{-1}_R(\Omega)$, $\pi^{-1}_S(\Omega)$ and $\pi^{-1}_T(\Omega)$ are trivial fibrations over $\Omega$;
    \item \label{item:OmegaR}The exceptional locus $\Omega_R$ of $\pi_R$ is isomorphic to $(\PP\Hom^{\omega}(W, V)\sslash SO(W))\times \Omega$.
    \item \label{item:transversesliceSigmaROmegaR}Let $I'_{2g-3}$ be the quotient of $I_{2g-3}$ by the involution which exchanges the coordinates $x_i$ and $y_i$. A slice normal to $\Sigma_R \cap \Omega_R \simeq (\PP\Hom^{\omega}_1(W, V)\sslash SO(W))\times \Omega \simeq \PP^{2g-1}\times \Omega$ in $\Omega_R$ is locally analytically isomorphic to an affine cone over $I'_{2g-3}$.
    \item \label{item:DeltaS} The singular locus $\Delta_S$ of $S$ is the strict transform of $(\PP\Hom^{\omega}_2(W, V)\sslash SO(W))\times \Omega \subseteq \Omega_R$ via $\pi_R$, which is isomorphic to a $\PP^2$-bundle over $\Gr^{\omega}(2,V)\times \Omega$.
    \item\label{item:normalsliceDeltaS} A slice normal to $\Delta_S$ in $S$ is locally analytically isomorphic to the quotient $\CC^{2g-3}/\pm 1$.
\end{enumerate}
\end{prop}
\begin{proof}
Since $\Omega(C, \Sl_2)$ is a collection of $2^{2g}$ points, \eqref{item:exclocusoverOmega} obviously holds. Consider now the \'{e}tale cover $\tau\colon M(C, \Sl_n) \times M(C, \Gl_1) \to M(C, \Gl_n)$ trivialising \eqref{albmap}. For some (or any) $x \in \Omega(C, \Sl_n)$ we have
\[\pi_{T(C, \Sl_n)}^{-1}(x) \times \Omega(C, \Gl_n) \simeq \pi_{T(C, \Sl_n)}^{-1}(x) \times M(C, \Gl_1) \simeq \pi_{T(C, \Gl_n)}^{-1}(\Omega(C, \Gl_n)),\]
since $\tau$ is \'{e}tale. Thus, \eqref{item:exclocusoverOmega} holds for $G=\Gl_2$ too.

\eqref{item:OmegaR}, \eqref{item:transversesliceSigmaROmegaR}, \eqref{item:DeltaS}, \eqref{item:normalsliceDeltaS} follow instead from \cref{prop:normalslices}.\eqref{item:localmodelOmega}, \cite[(1.7.12) and (1.7.16)]{OGrady99}, \cite[(3.5.1)]{OGrady99} and \cite[(3.5.1)]{OGrady99} respectively; see alternatively the proof of \cite[Proposition 3.2]{ChoyKiem07}.
\end{proof}


\begin{prop}[The exceptional divisors of $\pi_T$]\label{prop:exceptionaldivisors}
\begin{enumerate}
\item \label{item:D1} $D_1$ is a $\hat{\PP}^5$-bundle over $\Gr^{\omega}(3,V) \times \Omega$.
    \item \label{item:D20} 
    Let $\mathbb{I}_{2g-3}$ be the (Zariski locally trivial) $I_{2g-3}$-bundle in $\PP(\mathcal{E}) \times \PP(\mathcal{E}^*)$ over $\Sigma^{\circ}$. Then $D^{\circ}_2$ is the quotient of $\mathbb{I}_{2g-3}$ by the involution $\iota'$
  \begin{align*}
        \iota'\colon \PP(\mathcal{E}) \times \PP(\mathcal{E}^*) & \to \PP(\mathcal{E}) \times \PP(\mathcal{E}^*),\\
        (v, [x_i], [y_j]) & \mapsto (\iota(v), [y_j], [x_i]),
  \end{align*}
  extending the involution $\iota$ on $\Sigma^{\circ}$ defined in \cref{def:involution}.
   \item \label{item:D3} $D_3$ is a $\PP^{2g-4}$-bundle over a (Zariski locally trivial) $\PP^2$-bundle over $\Gr^{\omega}(2,V) \times \Omega$.
   \item \label{item:D13} $D_{13}$ is a $\hat{Q}$-bundle over $\Gr^{\omega}(3,V) \times \Omega$.
\end{enumerate}
\end{prop}
\begin{proof}
Let $\mathcal{U}_m$ be the universal bundle over $\Gr^{\omega}(m, V)$, with $m=2,3$, and $\Hom_k(W, \mathcal{U}_m)$ be the subbundle of $\Hom(W, \mathcal{U}_m)$ of rank $\leq k$. The quotient space $\PP \Hom_k(W, \mathcal{U}_m)\sslash SO(W)$ is isomorphic to the space of quadrics $\PP(S^2_k \mathcal{U}_m)$ of rank $\leq k$. There are obvious forgetful maps
\begin{align*}
    f_3\colon & \PP\Hom(W, \mathcal{U}_3) \to \PP\Hom^{\omega}(W,V)\\
    f_2 \colon & \PP\Hom_2(W, \mathcal{U}_2) \to \PP\Hom^{\omega}_2(W,V),\\
\end{align*}
 which induces the following diagrams
\[
\begin{tikzpicture}[baseline= (a).base]
\node[scale=0.9] (a) at (0,0){
\begin{tikzcd}
   \Bl_{\PP(S^2_1 \mathcal{U}_3)}\PP(S^2_2 \mathcal{U}_3)\times \Omega \ar[r, hook]  \arrow[rrr, bend left=15, "\simeq"]
   \arrow[ddr, "\hat{Q}-\text{bundle}"'] & \Bl_{\PP(S^2_1 \mathcal{U}_3)}\PP(S^2 \mathcal{U}_3)\times \Omega\ar[r, "\simeq"] \ar[d, "birat."'] &
    D_1 \ar[d, "birat."]  & D_{13} = \mathrm{Exc}(\pi_T|_{D_1}) \ar[l, hook'] \ar[d, "\PP^{2g-5}- \text{bundle}"]\\
    & \PP(S^2 \mathcal{U}_3) \times \Omega \ar[r, "f_3\sslash SO(W)"]\ar[d, "\PP^5-\text{bundle}"] &
    \Omega_S \ar[d, "birat."] & \Delta_S \ar[l, hook'] \\
    & \Gr^{\omega}(3, V) \times \Omega & \Omega_R. &
\end{tikzcd}
};
\end{tikzpicture}
\]
\[
\begin{tikzpicture}[baseline= (a).base]
\node[scale=0.9] (a) at (0,0){
\begin{tikzcd}
  & \quad \quad \quad & D_3 \ar[d, "\PP^{2g-4}- \text{bundle}"]\\
   \PP(S^2 \mathcal{U}_2)\times \Omega \ar[rr, "\simeq"] \ar[d, "\simeq"] &&
    \Delta_S \ar[d, "birat."]\\
     (\PP\Hom(W, \mathcal{U}_2)\sslash SO(W))  \times \Omega \ar[rr, "f_2\sslash SO(W)"]\ar[d, "\PP^2-\text{bundle}"] & &
    (\PP\Hom^{\omega}_2(W,V)\sslash SO(W)) \times \Omega \ar[d, "birat."]\\
     \Gr^{\omega}(2, V) \times \Omega & & \Omega_R. &
\end{tikzcd}
};
\end{tikzpicture}
\]
A proof of the isomorphisms above is provided in \cite[(3.1.1) and (3.5.1)]{OGrady99}; see alternatively \cite[Proposition 3.2]{ChoyKiem07}. This shows \eqref{item:D1}, \eqref{item:D3}, \eqref{item:D13}. To show \eqref{item:D20}, one can repeat the argument of \cite[Proposition 3.2.(2)]{ChoyKiem07} verbatim.
\end{proof}

\begin{prop}\label{prop;HodgeTateI} $\Gr^{\omega}(m,V)$ and the fibres of $\Delta_S$, $D_1$, $D_3$, $D_{13}$ and $\Omega_S$ over $\Omega$ have pure cohomology of Hodge-Tate type. In particular, they do not have odd cohomology. 
Their E-polynomials are
\begin{align*}
    E(\Gr^{\omega}(2,V)) & = \frac{(1-q^{2g-2})(1-q^{2g})}{(1-q)(1-q^2)}\\
    E(\Gr^{\omega}(3,V)) & = \frac{(1-q^{2g-4})(1-q^{2g-2})(1-q^{2g})}{(1-q)(1-q^2)(1-q^3)}\\
    E(\Delta_S) & = \frac{(1-q^3)(1-q^{2g-2})(1-q^{2g})}{(1-q)^2(1-q^2)} \cdot E(\Omega)\\
    E(D_1) & = \frac{(1-q^4)(1-q^{2g-4})(1-q^{2g-2})(1-q^{2g})}{(1-q)^3(1-q^2)} \cdot E(\Omega)\\
    E(D_3) & = \frac{(1-q^{3})(1-q^{2g-3})(1-q^{2g-2})(1-q^{2g})}{(1-q)^3 (1-q^2)} \cdot E(\Omega)\\
    E(D_{13}) & = \frac{(1-q^{3})(1-q^{2g-4})(1-q^{2g-2})(1-q^{2g})}{(1-q)^3 (1-q^2)} \cdot E(\Omega)\\
    E(\Omega_S) & =  E(D_1)-E(\Delta_S) \cdot \sum^{2g-6}_{i=0} q^{i+1} \\
    & = \frac{(1-q^{2g-2})(1-q^{2g-1})(1-q^{2g})}{(1-q)^2(1-q^2)} \cdot E(\Omega). 
\end{align*}
\end{prop}
\begin{proof}
Note that $\Gr^{\omega}(m,V)$ is a smooth projective variety whose cohomology is of Hodge-Tate type; see for instance \cite[Lemma 3.1]{ChoyKiem07}. Hence, the fibres of $\Delta_S$, $D_1$, $D_3$ and $D_1 \cap D_3$ over $\Omega$ have pure cohomology of Hodge-Tate type by \cref{prop:geometryblowup} and \ref{prop:exceptionaldivisors}. Since $\Omega_S$ has only quotient singularities, the natural inclusion $H^*(\Omega_S) \simeq IH^*(\Omega_S) \hookrightarrow H^*(D_1)$ implies that the fibres of $\Omega_S$ over $\Omega$ have pure cohomology of Hodge-Tate type. The computation of the E-polynomials follows immediately from \cite[Lemma 3.1]{ChoyKiem07}, \cref{prop:geometryblowup} and \ref{prop:exceptionaldivisors}, except maybe for $\Omega_S$. In that case, we use the decomposition theorem for the blow-up map $\pi_T|_{D_1}$, which actually reduces to \cite[Theorem 7.31]{Voisin07I}.
\end{proof}

\subsection{The incidence variety $I_{2g-3}$}\label{sec:incidencevariety}

The incidence variety $I_{2g-3} \subset \PP^{2g-3} \times \PP^{2g-3}$ is the projectivization of the vector bundle $\Omega^1_{\PP^{2g-3}}(1)$ over $\PP^{2g-3}$. Hence, we can write
\begin{equation}\label{cohomI2g-3}
    H^*(I_{2g-3}) = \QQ[a,b]/(a^{2g-2}, b^{2g-2}, b^{2g-3} + a b^{2g-4} + \ldots + a^{2g-4} b + a^{2g-3}),
\end{equation}
where $a$ and $b$ have degree $2$, and they are pullback of the first Chern classes of the tautological line bundle of $\PP^{2g-3}$ via the two projections $I_{2g-3} \subset \PP^{2g-3} \times \PP^{2g-3} \to \PP^{2g-3}$. Note that $I_{2g-3}$ has no odd cohomology.

The involution which exchanges the factors of the product $\PP^{2g-3} \times \PP^{2g-3}$ leaves $I_{2g-3}$ invariant, and in cohomology exchanges the classes $a$ and $b$. Consider the decomposition into eigenspaces for the involution (relative to eigenvalues $\pm 1$ respectively)
\[H^*(I_{2g-3}) = H^*(I_{2g-3})^+ \oplus H^*(I_{2g-3})^-.\]
For $d=2k<4g-7$, we have
\begin{align*}
 H^d(I_{2g-3})^+ = H^d(\PP^{2g-3}\times \PP^{2g-3})^+ = \langle a^ib^j | i+j=d\rangle^+ = \langle a^ib^j + a^j b^i| i+j=d\rangle.
 \end{align*}
Therefore, we obtain that
\[
\dim H^{2k}(I_{2g-3})=k+1, \,\, \dim H^{2k}(I_{2g-3})^+=\left\lceil \frac{k+1}{2} \right\rceil, \,\, \dim H^{2k}(I_{2g-3})^-=\left\lceil \frac{k}{2} \right\rceil.
\]
\begin{prop} Setting $q \coloneqq t^2=uv$, the Poincar\'{e} polynomials (equivalently E-polynomials) of $I_{2g-3}$, of the invariant and variant part of its cohomology are
\begin{align}
    \label{item:PPI2g-3}P_t(I_{2g-3}) & = E(I_{2g-3})  = \frac{(1-q^{2g-2})(1-q^{2g-3})}{(1-q)^2}\\
   \label{item:PPI2g-3+} P_t(I_{2g-3})^+ & = E(I_{2g-3})^+  
   = \frac{(1-q^{2g-2})^2}{(1-q^2)(1-q)}\\
    \label{item:PPI2g-3-} P_t(I_{2g-3})^- & = E(I_{2g-3})^-  
    = q \frac{(1-q^{2g-2})(1-q^{2g-4})}{(1-q^2)(1-q)}
\end{align}
\end{prop}
\begin{proof}
$I_{2g-3}$ is a $\PP^{2g-4}$-bundle over $\PP^{2g-3}$, and this gives \eqref{item:PPI2g-3}. We now estimate $P_t(I_{2g-3})^+ - P_t(I_{2g-3})^-$. For $d=2k<4g-7$, we have
\[\dim H^{2k}(I_{2g-3})^+ - \dim H^{2k}(I_{2g-3})^- = \left\lceil \frac{k+1}{2} \right\rceil- \left\lceil \frac{k}{2} \right\rceil = \begin{cases}
1 & \text{ for } k=2l,\\ \label{equation:computation primitivepart}
0  & \text{ for } k=2l+1.
\end{cases}\]
Since the polarization $\mathcal{O}(1,1)$ is $\iota$-invariant, the hard Lefschetz theorem gives $\dim H^{d}(I_{2g-3})^{\pm} = \dim H^{8g-14 - d}(I_{2g-3})^{\pm}$. Thus, we can write
\begin{equation}\label{eq:P+-P-}
 P_t(I_{2g-3})^+ - P_t(I_{2g-3})^- = \sum^{g-2}_{l=0}q^{2l} + q^{2g-3} \sum^{g-2}_{l=0}q^{2l} = \frac{(1-q^{2g-2})(1+q^{2g-3})}{(1-q^2)}.   
\end{equation}
Finally, substituting \eqref{item:PPI2g-3} and \eqref{eq:P+-P-} in
\begin{align*}
    P_t(I_{2g-3})^+ & = \frac{1}{2}(P_t(I_{2g-3}) + ( P_t(I_{2g-3})^+ - P_t(I_{2g-3})^- ))\\
    P_t(I_{2g-3})^- & =  P_t(I_{2g-3}) - P_t(I_{2g-3})^+,
\end{align*}
we obtain \eqref{item:PPI2g-3+} and \eqref{item:PPI2g-3-}.
\end{proof}

\subsection{Intersection cohomology of local models}\label{sec:intersectioncohomologyoflocalmodels}
The goal of this section is to compute the intersection cohomology of the normal slices $N_{\Sigma}$ and $N_{\Omega}$. This is an important step to determine the summands of the decomposition theorem in \cref{introthm:decompositiontheorempiT}.
\begin{prop}\label{prop:NSIGMA}
Let $N_{\Sigma}$ be a slice normal to $\Sigma$ at a point in $\Sigma \setminus \Omega$. Then
\begin{equation}\label{intersection cohomology NSigma}
IH^d(N_{\Sigma})  \simeq 
\begin{cases}
H^{d}_{\mathrm{prim}}(I_{2g-3}) \simeq \QQ & \text{ for } d=2k<4g-6,\\
0 & \text{ otherwise.}\\
\end{cases}
\end{equation}
\begin{equation}\label{intersection cohomology NSigma: var and invar}
IH^{2k}(N_{\Sigma})  \subseteq 
\begin{cases}
H^{2k}(I_{2g-3})^+ & \text{ for } k=2l,\\
H^{2k}(I_{2g-3})^- & \text{ for } k=2l+1.
\end{cases}
\end{equation}
\end{prop}
\begin{proof}
 \eqref{intersection cohomology NSigma} follows from \cref{prop: interseccone}, since $N_{\Sigma}$ is locally isomorphic to an affine cone over the smooth variety $I_{2g-3}$ (with $\iota$-invariant conormal bundle $\mathcal{O}(1,1)$) by \cref{prop:normalslices}.\eqref{item:localmodelSigma}. 

Since $\mathcal{O}(1,1)$ is $\iota$-invariant, $H^d_{\mathrm{prim}}(I_{2g-3})$ is $\iota$-invariant too. Hence, \eqref{intersection cohomology NSigma: var and invar} follows from the following dimensional argument
\begin{align*}
1 \geq \dim IH^{2k}(N_\Sigma)  \geq \dim IH^{2k}(N_\Sigma)^+ & =  \dim H^{2k}(I_{2g-3})^+ - \dim H^{2k-2}(I_{2g-3})^+ \nonumber \\ 
& = \left\lceil \frac{k+1}{2} \right\rceil- \left\lceil \frac{k}{2} \right\rceil = \begin{cases}
1 & \text{ for } k=2l,\\ \label{equation:computation primitivepart}
0  & \text{ for } k=2l+1.
\end{cases}
\end{align*}
\end{proof}

\begin{prop}\label{prop:NSIGMA'}
Let $N_{\Sigma_R \cap \Omega_R}$ be a slice normal to $\Sigma_R \cap \Omega_R$ in $\Omega_R$. Then
\begin{equation}\label{intersection cohomology NSigma'}
IH^d(N_{\Sigma_R \cap \Omega_R})  \simeq 
\begin{cases}
H^d_{\mathrm{prim}}(I'_{2g-3}) \simeq \QQ & \text{ for } d=4k<4g-6,\\
0 & \text{ otherwise}.
\end{cases}
\end{equation}
In particular,
\begin{equation}\label{dimintersectioncohomologyNSigma'}
\dim H^{4g-6+2i}(I'_{2g-3}) - \dim IH^{4g-6+2i}(N_{\Sigma_R \cap \Omega_R})  = \left\lceil \frac{2g-3-|i|}{2}\right\rceil.
\end{equation}
\end{prop}

\begin{proof}
\cref{prop: interseccone} and \ref{prop:geometryblowup}.\eqref{item:transversesliceSigmaROmegaR} gives \eqref{intersection cohomology NSigma'}, while
\eqref{dimintersectioncohomologyNSigma'} follows immediately from $ H^{2k}(I'_{2g-3}) = H^{2k}(I_{2g-3})^+ $. 
\end{proof}

\begin{prop}\label{prop:IH(OmegaR)}
The intersection E-polynomial of $\Omega_R$ is 
\[IE(\Omega_R) = \frac{(1-q^{4g-4})(1-q^{2g})}{(1-q)(1-q^2)}\cdot E(\Omega).\]
\end{prop}
\begin{proof}
We apply the decomposition theorem to the restriction of $\pi_S$ to the strict transform $\Omega_S \coloneqq \pi^{-1}_{S, *} \Omega_R$. 

By \cref{prop:geometryblowup}.\eqref{item:transversesliceSigmaROmegaR} the defect of semismallness of $\pi_S|_{\Omega_S}$ is 
\begin{align*}
 r(\pi_S|_{\Omega_S}) & \coloneqq \dim \Omega_S \times_{\Omega_R} \Omega_S - \dim \Omega_R \\
 & = 2 \dim I'_{2g-3} + \dim \Sigma_ R \cap \Omega_R - \dim \Omega_R = 4g-8,
\end{align*}
and $\Sigma_ R \cap \Omega_R$ is the only support of the decomposition theorem for $\pi_S|_{\Omega_S}$. Note that $R^i \pi_{S, *} \QQ_{\pi^{-1}_S(\Sigma_ R \cap \Omega_R)}$ are trivial local systems over $\Sigma_ R \cap \Omega_R \simeq \PP^{2g-1} \times \Omega$, because of \cref{prop:geometryblowup}.\eqref{item:exclocusoverOmega} and the simple connectedness of $\PP^{2g-1}$. Hence, there exist integers $a(i)$ such that
\begin{align*}
  R(\pi_{S}|_{\Omega_S})_* \QQ [\dim \Omega_S] & = \bigoplus^{4g-8}_{i=-4g+8} {^{\mathfrak{p}}}\mathcal{H}^i((\pi_{S}|_{\Omega_S})_* \QQ[\dim \Omega_S])[-i]\\  
  & = IC_{\Omega_R} \oplus \bigoplus^{2g-4}_{i=-2g+4} \QQ^{a(i)}_{\Sigma_ R \cap \Omega_R}[\dim\Sigma_ R \cap \Omega_R-2i](-2g+3-i).
\end{align*}
At the stalk level, at $x \in \Sigma_R \cap \Omega_R$, we obtain by \eqref{dimintersectioncohomologyNSigma'}
\[
a(i) = \dim H^{4g-6+2i}(I'_{2g-3})-\dim IH^{4g-6+2i}(N_{\Sigma_R \cap \Omega_R}) = \left\lceil \frac{2g-3-|i|}{2}\right\rceil.
\]
Together with \cref{prop;HodgeTateI}, we get
\begin{align*}
 & IE(\Omega_R) = E(\Omega_S) - E(\Omega) \cdot E(\PP^{2g-1}) \cdot \sum^{2g-4}_{i=-2g+4}q^{2g-3+i} \left\lceil \frac{2g-3-|i|}{2}\right\rceil\\
& = \left( \frac{(1-q^{2g-2})(1-q^{2g-1})(1-q^{2g})}{(1-q)^2(1-q^2)} - \frac{q (1 - q^{2 g - 3}) (1 - q^{2 g - 2})(1-q^{2g})}{(1-q)^2 (1-q^2)} \right) \cdot E(\Omega)\\
& = \frac{(1-q^{4g-4})(1-q^{2g})}{(1-q)(1-q^2)} \cdot E(\Omega).
\end{align*}

\end{proof}
\begin{prop}\label{prop:IH(U)Omega} Let $N_{\Omega}$ be a slice normal to $\Omega$. Then $IH^*(N_{\Omega})$ is pure of Hodge-Tate type with intersection Poincar\'{e} polynomial (equivalently intersection E-polynomials)
\[IP_t(N_{\Omega}) = IE(N_{\Omega}) = \frac{1-q^{2g}}{1-q^2}.\] 
\end{prop}
\begin{proof} Since $IH^*(N_{\Omega}) \hookrightarrow IH^*(\pi_R ^{-1}(x))$ for some $x \in \Omega$, $IH^*(N_{\Omega})$ is pure of Hodge-Tate type by \cref{prop;HodgeTateI}. 

Recall now that $N_{\Omega}$ is an affine cone over $\Omega_R$ by \cref{prop:singulMB}.\eqref{item:localmodelOmega}.
Hence, \cref{prop: interseccone} implies that the intersection Poincar\'{e} polynomial $IP_t(N_{\Omega})$ is a polynomial in the variable $q = t^2$ of degree at most $3g-4$, given by
\begin{align*}
IE(N_{\Omega})&=[IE(\pi_R ^{-1}(x))-q IE(\pi_R ^{-1}(x))]_{\leq 3g-4}= \frac{1}{E(\Omega)}[IE(\Omega_R)-q IE(\Omega_R)]_{\leq 3g-4} \\
& = \bigg[ \frac{(1-q^{4g-4})(1-q^{2g})}{(1-q^2)}\bigg]_{\leq 3g-4} = \frac{1-q^{2g}}{1-q^2}.
\end{align*}
\end{proof}

\section{Decomposition theorem}\label{sec:decompositionthm}
\begin{proof}[Proof of \cref{introthm:decompositiontheorempiT}]
Let $\pi^{\circ}_T$ 
be the restriction of $\pi_T$ over $M^{\circ} \coloneqq M \setminus \Omega$.
By \cref{prop:singulMB}.\eqref{item:localmodelSigma} the defect of semismallness of $\pi^{\circ}_T$ is 
\begin{align*}
 r(\pi^{\circ}_T) & \coloneqq \dim \pi^{-1}_T(M^{\circ}) \times_{M^{\circ}} \pi^{-1}_T(M^{\circ}) - \dim M \\
 & = 2 \dim I_{2g-3} + \dim \Sigma - \dim M = 4g-8,
\end{align*}
and $\Sigma^\circ$ is the only support of the decomposition theorem for $\pi^{\circ}_T$. Hence, there exists a splitting
\[
R \pi^{\circ}_{T, *} \QQ[\dim T]= IC_{M^\circ} \oplus \bigoplus^{2g-4}_{i=-2g+4}\mathscr{L}_i[\dim \Sigma -2i]
\]
for some semisimple local systems $\mathscr{L}_i$ supported on $\Sigma^\circ$. Restricting to $D^\circ_2 = \pi^{-1}_T(\Sigma^{\circ})$, we obtain 
\[
R (\pi^{\circ}_{T|_{D^\circ_2}})_* \QQ[\dim T]= IC_{M^\circ}|_{D^{\circ}_2} \oplus \bigoplus^{2g-4}_{i=-2g+4}\mathscr{L}_i[\dim \Sigma -2i].
\]
By \cref{prop:exceptionaldivisors}.\eqref{item:D20} there exists a commutative square
\[
\begin{tikzpicture}[baseline= (a).base]
\node[scale=1] (a) at (0,0){
\begin{tikzcd}
   \mathbb{I}_{2g-3}\ar[r, "q'"] \ar[d, "p_2"'] & D^\circ_2 \ar[d, "\pi^{\circ}_{T|D^\circ_2}"] \\
  (\CC^*)^{2g, \circ} \ar[r, "q"] &
    \Sigma^{\circ}, 
\end{tikzcd}
};
\end{tikzpicture}
\]
where the horizontal arrows are \'{e}tale double covers, $p_2$ is a Zariski locally trivial fibration with fibre $I_{2g-3}$, and  $(\CC^*)^{2g, \circ}$ is the complement in $(\CC^*)^{2g}$ of the locus fixed by the involution $v \mapsto -v$. Taking cohomology, we write
\begin{align*}
  R (\pi^{\circ}_{T|_{D^\circ_2}})_* \QQ  & = R (\pi^{\circ}_{T|_{D^\circ_2}})_* \circ (Rq'_* \QQ)^+  = (R(\pi^{\circ}_{T|_{D^\circ_2}} \circ q')_* \QQ)^+ = (R(q \circ p_2)_* \QQ)^+\\
  & = \left(q_* \bigoplus^{4g-7}_{i=0} (\QQ_{(\CC^*)^{2g, \circ}} \otimes H^{2i}(I_{2g-3}))[-2i]\right)^+\\
& = \bigoplus^{2g-4}_{i=0} \QQ_{\Sigma^{\circ}}^{\lceil \frac{i+1}{2}\rceil}[-2i](-i) \oplus \bigoplus^{2g-4}_{i=0} \mathscr{L}^{\lceil \frac{i}{2}\rceil}[- 2  i](-i) \\
& \quad \oplus \bigoplus^{4g-7}_{i=2g-3} \QQ_{\Sigma^{\circ}}^{\lceil 2g-3 - \frac{i}{2}\rceil}[-2i](-i) \oplus \bigoplus^{4g-7}_{i=2g-3} \mathscr{L}^{\lceil 2g-3 - \frac{i+1}{2}\rceil}[-2i](-i),
\end{align*}
where $\mathscr{L}$ is the rank-one local system defined in \cref{def:involution}.
Together with \cref{prop:NSIGMA}, we obtain 
\[
R \pi^{\circ}_{T, *} \QQ[\dim T]= IC_{M^\circ} \oplus \bigoplus^{2g-4}_{i=-2g+4}(\QQ_{\Sigma^\circ}^{\lceil \frac{2g-3-|i|}{2} \rceil}\oplus \mathscr{L}^{\lfloor \frac{2g-3-|i|}{2} \rfloor})[\dim \Sigma -2i](-2g+3-i).
\]

This splitting holds on $M^{\circ}$, and we now extend it through $\Omega$. Note that the defect of semismallness of $\pi_T$ is 
\begin{align*}
 r(\pi_T) & = \dim \pi^{-1}_T(M) \times_{M} \pi^{-1}_T(M) - \dim M \\
 & = 2 \dim D_1 - \dim \Omega - \dim M = 6g-8.
\end{align*}
Since $\Sigma$ is a rational homology manifold, $IC_{\Sigma}(\QQ_{\Sigma^\circ}) \simeq \QQ_{\Sigma}[\dim \Sigma]$. Further, the definition of $\mathscr{L}$ yields $IC_{\Sigma}(\mathscr{L}) = i_{\Sigma^\circ,*}\mathscr{L}[\dim \Sigma]$. Therefore, there exists integers $b(j)$ such that
\begin{align*}
R\pi_{T, *}\QQ[\dim T]= IC_{M} & \oplus  \bigoplus^{2g-4}_{i=-2g+4}\QQ_{\Sigma}^{\lceil \frac{2g-3-|i|}{2} \rceil}[\dim \Sigma -2i](-2g+3-i) \\
& \oplus \bigoplus^{2g-4}_{i=-2g+4} i_{\Sigma^\circ,*}\mathscr{L}^{\lfloor \frac{2g-3-|i|}{2} \rfloor}[\dim \Sigma-2i](-2g+3-i)\\
& \oplus \bigoplus^{3g-4}_{j=-3g+4} \QQ^{b(j)}_{\Omega}[\dim \Omega -2j](-3g+3-j),    
\end{align*}
Localising at $x \in \Omega$, we obtain
\begin{small}
\[
\dim H^{2d}(\pi^{-1}_T(x))  = \begin{cases}
\dim IH^{2d}(N_{\Omega}) + \left\lceil \frac{2g-3-|d-2g+3|}{2} \right\rceil + b(d-3g+3) & \text{ for } 0 \leq d < 4g-6,\\
\dim IH^{2d}(N_{\Omega}) + b(d-3g+3)& \text{ otherwise}.\\
\end{cases}
\]
\end{small}
Therefore, $b(j)$ is the coefficient of $q^{d-3g+3}$ of the polynomial
\begin{align*}
 E & (\pi^{-1}_T(x))  - IE(N_{\Omega}) - \sum^{2g-3}_{i=-2g+3}\left\lceil \frac{2g-3-|i|}{2} \right\rceil q^{2g-3+i} = \\
 & = \frac{1}{E(\Omega)}(E(D_1)+E(D_3)-E(D_{13})) - IE(N_{\Omega}) - \frac{q (1 - q^{2 g - 3}) (1 - q^{2 g - 2})}{(1-q) (1-q^2)} \\
 & = \frac{\left(1-q^{2g-2}\right) \left(1-q^{2g}\right) \left(1-q^4 -q^{2
 g-3}-q^{2g-1}+2q^{2g}\right)}{\left(1-q\right)^3 
 \left(1-q^2\right)}-\frac{1-q^{2g}}{1-q^2}\\
 & -\frac{q (1 - q^{2 g - 3}) (1 - q^{2 g - 2})}{(1-q) (1-q^2)}.
\end{align*}
\end{proof}

\subsection{Applications of the decomposition theorem}\label{sec:applicationdthm}
\label{sec:proofs}
\begin{proof}[Proof of \cref{thm:IEmainformula}]
Taking cohomology with compact support, \cref{introthm:decompositiontheorempiT} gives
\begin{align*}
IE(M)  & = E(T)- E(\Sigma_{\iota})^+ \cdot \left( \sum^{2g-3}_{i=-2g+3} \left\lceil \frac{2g-3-|i|}{2} \right\rceil q^{2g-3+i}\right)\\ 
  & - E(\Sigma_{\iota})^- \cdot \left(\sum^{2g-3}_{i=-2g+3} \left\lfloor \frac{2g-3-|i|}{2} \right\rfloor q^{2g-3+i} \right)\\
  & - E(\pi^{-1}_T(\Omega)) + E(\Omega) \cdot IE(N_{\Omega}) + E(\Omega) \cdot \sum^{2g-3}_{i=-2g+3}\left\lceil \frac{2g-3-|i|}{2} \right\rceil q^{2g-3+i}\\
   & = E(M) - E(\Sigma) +E(D^{\circ}_{2})\\
  & - (E(\Sigma_{\iota})^+-E(\Omega))\cdot q \frac{(1 - q^{2 g - 3}) (1 - q^{2 g - 2})}{(1-q) (1-q^2)} \\
  & - E(\Sigma_{\iota})^- \cdot q^2 \frac{(1 - q^{2 g - 4}) (1 - q^{2 g - 3})}{(1-q) (1-q^2)} +E(\Omega) \cdot IE(N_{\Omega}).
  \end{align*}
  Now \cref{prop:exceptionaldivisors}.\eqref{item:D20} gives
 \begin{equation}\label{eq:Ecirc}
    E(D^{\circ}_{2}) = E(I_{2g-3})^+ \cdot (E(\Sigma_{\iota})^+-E(\Omega)) +  E(I_{2g-3})^- \cdot E(\Sigma_{\iota})^-. 
 \end{equation}
    Therefore, we obtain
  \begin{align*}
  IE(M)& = E(M) - E(\Sigma) + (E(\Sigma_{\iota})^+-E(\Omega))\cdot \frac{1 - q^{2 g - 2}}{ 1-q^2} \\
  & + E(\Sigma_{\iota})^- \cdot q\frac{1 - q^{2 g - 4} }{1-q^2}+ E(\Omega) \cdot \frac{1-q^{2g}}{1-q^2}\\
  & = E(M)+ (q^2E(\Sigma_{\iota})^+ + qE(\Sigma_{\iota})^-) \cdot \frac{1 - q^{2 g - 4}}{ 1-q^2} +E(\Omega) \cdot q^{2g-2}.
\end{align*}
\end{proof}
The variant and anti-invariant E-polynomials of $(\CC^*)^{2g}$ and $\TJac$ with respect to the involution $\iota$ defined in \cref{def:involution} are
\begin{align}
\label{eq:C+}E((\CC^*)^{2g})^+ & =  \sum^{g}_{d=0} \dim (\Lambda^{2d}V) q^{2d}= \frac{1}{2}((1-q)^{2g}+(1+q)^{2g}),\\ 
\label{eq:C-}E((\CC^*)^{2g})^- & =  -\sum^{g-1}_{d=0} \dim (\Lambda^{2d+1}V) q^{2d+1}= \frac{1}{2}((1-q)^{2g}-(1+q)^{2g}),\\
\label{eq:ETJAC+}E(\TJac)^+  & = \frac{1}{2}(uv)^g \left((1-u)^g(1-v)^g+(1+u)^g(1+v)^g\right),\\ 
\label{eq:ETJAC-}E(\TJac)^- & = \frac{1}{2}(uv)^g \left((1-u)^g(1-v)^g-(1+u)^g(1+v)^g\right).   
\end{align}

\begin{proof}[Proof of \cref{thm:IEB}]
We compute $IE(\MB(C, \Sl_2))$ and $IE(\MB(C, \Gl_2))$ from \eqref{eq:IE} by substitution. To this end, recall that $E(\MB(C, \Sl_2))$ and $E(\MB(C, \Gl_2))$ have been computed in \cite[Theorem 1.3]{BaragliaHekmati17} (equiv. \cite[Theorem 2]{MartinezMunoz16}). The E-polynomials of $\Sigma_{\iota, B}$ and $\Omega_B$ instead can be determined by using the description of these loci in \cref{prop:singulMB}, together \eqref{eq:C+} and \eqref{eq:C-}.
Further, by \eqref{eq:PGLtoSL} we have $IE(\MB(C, \PGl_2)) = (1-q)^{-2g} \cdot IE(\MB(C, \Gl_2))$.
\end{proof}

\begin{proof}[Proof of Theorem \ref{thm:IEDol}] 
In view of \[E(M)=E(M^{\mathrm{sm}})+E(\Sigma)=E(M^{\mathrm{sm}})+E(\Sigma_{\iota})^+,\] 
we obtain \cref{thm:IEDol} simply by substituting \eqref{eq:ETJAC+} and \eqref{eq:ETJAC-} in \eqref{eq:IE}.
\end{proof}

\begin{proof}[Proof of \cref{thm:IP}]
The purity of $IH^*(\MDol)$ (\cref{prop:purityandother}) and the Poincar\'{e} duality give 
\[ IP_t(M) = t^{2\dim M} IE(\MDol; -t^{-1}, -t^{-1}).\]
\cref{thm:IP} then follows from \cref{thm:IEDol} and elementary algebraic manipulations.
\end{proof}

\begin{proof}[Proof of \cref{thm:EDolT}] By the additivity of the E-polynomial, we have
\begin{align*}
E(T)=E(M^{\mathrm{sm}})+E(D_{1})+E(D^{\circ}_{2})+E(D_{3})-E(D_{13}).    
\end{align*}
The formula \eqref{eq:Ecirc}, together with \eqref{eq:C+}, \eqref{eq:C-}, \eqref{eq:ETJAC+} and \eqref{eq:ETJAC-},  yields
\begin{align*}
    E(D^{\circ}_{2, \mathrm{B}})
    & = \frac{1-q^{2g-2}}{2(1-q)} \big((1-q^{2g-3})(1-q)^{2g-1} + (1+q^{2g-3})(1+q)^{2g-1}\big)\\ 
    & -2^{2g}\frac{(1-q^{2g-2})^2}{(1-q^2)(1-q)},\\
    E(D^{\circ}_{2, \mathrm{Dol}})
    & = \frac{1}{2}(uv)^g \frac{1-(uv)^{2g-2}}{1-uv}\bigg( \frac{(1-u)^g(1-v)^g(1-(uv)^{2g-3})}{1-uv} \\
    & + \frac{(1+u)^g(1+v)^g(1+(uv)^{2g-3})}{1+uv}\bigg)-2^{2g} \frac{(1-(uv)^{2g-2})^2}{(1-(uv)^2)(1-uv)},
\end{align*}
and \cref{prop;HodgeTateI} gives
\[E(D_1)+E(D_3)-E(D_{13})=2^{2g}\frac{\left(1-q^{2g-2}\right) \left(1-q^{2g}\right) \left(1-q^4 -q^{2
 g-3}-q^{2g-1}+2q^{2g}\right)}{\left(1-q\right)^3 
 \left(1-q^2\right)}.\]
\end{proof}

\subsection{From $\Sl_2$ to $\PGl_2$ or $\Gl_2$}\label{sec:SLtoPGL}\label{sec:SL2PGL2GL2}
In order to compute $IE(\MDol(C, G))$ or $IP_t(M(C, G))$ for $G=\PGl_2, \Gl_2$, one can repeat the arguments for $\Sl_2$, and realise that in practise one can obtain the polynomials for $\PGl_2$ by replacing the coefficients $2^{2g}$ with $1$ in the corresponding polynomials for $\Sl_2$, as explained below. Further, one can use \eqref{eq:GLtoSL} and \eqref{eq:PGLtoSL} to write the polynomials for $\Gl_2$ from the $\PGl_2$ counterparts.
\begin{defn}
$\mathrm{Var}^{\Gamma}$ is the category of algebraic varieties endowed with a $\Gamma$-action, with $\Gamma$-equivariant morphisms as morphisms.

$\Gamma\mathrm{-mHs}$ is the abelian category whose objects are $\Gamma$-modules over $\QQ$ endowed with a mixed Hodge structure, and whose morphisms are $\Gamma$-equivariant morphisms of mixed Hodge structures. 
\end{defn}
\begin{defn}
Let $X$ be an algebraic variety endowed with an algebraic $\Gamma$-action. The \textbf{virtual Hodge realization} of $(X; \Gamma {\curvearrowright} X)$ is the element in the Groethendieck ring $K_0(\Gamma\mathrm{-mHs})$ defined by the formula
\[\chi_{\mathrm{Hdg}; \Gamma}(X)=\sum_k (-1)^k [H^k_c(T(C, \Sl_2)); \rho_M \colon \Gamma \to \operatorname{Aut}(H^k_c(X))].\]
The morphism
$\chi_{\mathrm{Hdg}; \Gamma}(\cdot)\colon K_0(\mathrm{Var}^{\Gamma}) \to K_0(\Gamma\mathrm{-mHs})$
is additive.
\end{defn}

The same Hodge realization was considered in \cite[\S 4]{HauselThaddeus03}, when Hausel and Thaddeus defined E-polynomials with coefficient in the characters of the finite abelian group $\Gamma$. 

Now consider the $\Gamma$-invariant stratification of $\TDol(C, \Sl_2)$ whose strata are
\begin{enumerate} \setcounter{enumi}{0}
    \item $S_0 \simeq \{(E, \phi) \in \MDol(C, \Sl_2) |\, E \text{ is stable}\}$;
    \item $S_1 \simeq \{(E, \phi) \in \MDol(C, \Sl_2) |\, E \simeq L \oplus L^{-1}, \, L \in \operatorname{Jac}(C), \, L \neq L^{-1}\}$;
    \item $S_2 \simeq \{(E, \phi) \in \MDol(C, \Sl_2) |\, E \text{ is a non-trivial extension of }L^{-1}\text{ by }L \text{ for } L\in \operatorname{Jac}(C)\text{ with } \, L \neq L^{-1}\}$;
    \item $S_3 \simeq \{(E, \phi) \in \MDol(C, \Sl_2) |\, E \simeq L \oplus L, \, L \in \operatorname{Jac}(C)\}$;
    \item $S_4 \simeq \{(E, \phi) \in \MDol(C, \Sl_2) |\, E \text{ is a non-trivial extension of }L\text{ by }L \text{ for } L\in \operatorname{Jac}(C)\}$;
    \item $S_5 \simeq \{(E, \phi) \in \MDol(C, \Sl_2) |\, E \text{ is unstable}\}$;
    \item $S_6 = D^{\circ}_2$, $S_7 = D_3 \setminus D_{13}$ and $S_8 = D_1$.
\end{enumerate}
This is indeed a stratification of $\TDol(C, \Sl_2)$, since 
$\MDol(C, \Sl_2)^{\text{sm}} = \bigsqcup^5_{i=0} S_i \simeq \TDol(C, \Sl_2) \setminus (D_1 \cup D_2 \cup D_3)$ by \cite[Example 3.13]{Hitchin1987} and \cref{prop:singulMB}.
The additivity of the virtual realization implies that
\[\chi_{\mathrm{Hdg}; \Gamma}(\TDol(C, \Sl_2))= \sum^8_{i=0} \chi_{\mathrm{Hdg}; \Gamma}(S_i).\]
By direct inspection (see \cite[\S 3]{KiemYoo08}) one can check that there exist algebraic varieties $Z_{ij}$ endowed with a $\Gamma$-action such that:
\begin{enumerate}
    \item the $\Gamma$-module $H^k_c(Z_{ij})$ is isomorphic to the direct sum of copies of the trivial and of the regular representation of $\Gamma$, i.e.\ there exists integers $l_{ijk}$ and $m_{ijk}$ such that there exists a $\Gamma$-equivariant isomorphism
    \[H^k_c(Z_{ij}) \simeq V^{\oplus n_{ijk}}_{\text{tr}} \oplus V^{\oplus m_{ijk}}_{\text{reg}},\]
    where $\Gamma$ acts trivially on $V_{\text{tr}} \simeq \QQ$, and via the regular representation on $V_{\text{reg}}\simeq \QQ^{2^{2g}}$. We call $V^{\oplus m_{ijk}}_{\text{reg}}$ the \textbf{regular part} of $H^k_c(Z_{ij})$.
    \item we have 
    \[\chi_{\mathrm{Hdg}; \Gamma}(S_i) = \sum_{j, k} \epsilon_{ij} (-1)^k [H^k_c(Z_{ij}); \rho_{ij}\colon \Gamma \to \operatorname{Aut}(H^k_c(Z_{ij}))],\]
where $\epsilon_{ij}$ is $\pm 1$, and $\rho_{ij}$ is a direct sum of copies of the trivial and/or of the regular representation.
\end{enumerate}
Denote by $E_{\text{reg}}(Z_{ij})$ the E-polynomial of the regular part of $H^*_c(Z_{ij})$, and let $E_{\text{tr}}(Z_{ij})\coloneqq E(Z_{ij}) - E_{\text{reg}}(Z_{ij})$. Then we have
\begin{align}
E(\TDol(C, \Sl_2)) & = \sum_{i,j}\epsilon_{ij} E_{\text{tr}}(Z_{ij}) + \sum_{ij}\epsilon_{ij} E_{\text{reg}}(Z_{ij}) \label{eq:Decompositionisotipiccomp}\\
& = \sum_{ij}\epsilon_{ij} E_{\text{tr}}(Z_{ij}) + 2^{2g} \sum_{ij}\epsilon_{ij} E_{\text{reg}}(Z_{ij})^{\Gamma} \nonumber  \\ 
E(\TDol(C, \Sl_2))^{\Gamma} & = \sum_{ij}\epsilon_{ij} E_{\text{tr}}(Z_{ij}) + \sum_{ij}\epsilon_{ij} E_{\text{reg}}(Z_{ij})^{\Gamma}.   \nonumber  
\end{align}
Via the decomposition theorem, the same holds for $IE(\MDol(C, \Sl_2))$, and so for $IP_t(M(C, \Sl_2))$, by the purity of $IH^*(M(C, \Sl_2))$, as explained in the proof of \cref{thm:IP}.
\begin{rmk}\label{rmk:replace2gwith1} Since in our case the varieties $Z_{ij}$ are completely explicit, we can check that all the coefficients $2^{2g}$ in \cref{thm:IEDol} comes from the E-polynomial of the regular part. So we obtain $IE(\MDol(C, \Sl_2))^{\Gamma}$ by replacing $2^{2g}$ by $1$. By \eqref{eq:GLtoSL}  and \eqref{eq:PGLtoSL} this gives $IE(\MDol(C, G))$ with $G=\PGl_2, \Gl_2$. Analogously, knowing $IP_t(M(C, \Sl_2))$, $E(\TDol(C, \Sl_2))$ and $P_t(T(C, \Sl_2))$ we can write their invariant counterparts, as well as their variants for $G=\PGl_2, \Gl_2$. 
\end{rmk}

\begin{rmk}
By the vanishing of the odd part of $IH^*_{\mathrm{var}}(M(C, \Sl_2))$ (cf \cref{thm: varinvarIE}), every non-trivial $(\Gr^W_{r+s} IH^d_{\mathrm{var}}(M(C, \Sl_2))^{r,s}$ will contribute with non-negative coefficient to $IE_{\mathrm{var}}(M(C, \Sl_2))$. Therefore, there is no cancellation, and the $\Gamma$-modules $IH^*(M(C, \Sl_2))$ and $H^*(T(C, \Sl_2))$ are direct sum of copies of the trivial and of the regular representation of $\Gamma$ by \eqref{eq:Decompositionisotipiccomp}. Comparing with \cite{DaskalopoulosWentworth10}, one can check that the same holds for $H^*(M(C, \Sl_2))$.
\end{rmk}

\section{P=W conjectures}\label{sec:P=Wconjectures}
\subsection{P=W conjecture for twisted character varieties}\label{sec:twistedcharacter}
The computation of E-polynomials of character varieties have been initiated in \cite{HauselRodriguez-Villegas2008}, for twisted character varieties $\MBtw=\MBTW$ 
\begin{align*} \MBTW \coloneqq \big\{ (A_1, B_1, \ldots, A_{g}, B_{g}) \in G^{2g} \, \big| \, \prod^{g}_{j=1}[A_j, B_j]=e^{2\pi i/d}\cdot 1_G \big\}\sslash G. 
\end{align*}
with $G=\Gl_n, \Sl_n$ or $\PGl_n$ and $\mathrm{gcd}(n,d)=1$; see also \cite{Mereb2015}. 

As in the untwisted case, a non-abelian Hodge correspondence holds for $\MBtw$: there exists a diffeomorphism $\Psi\colon \MDoltw \to \MBtw$, from the Dolbeault moduli space $\MDoltw=\MDolTW$ of semistable $G$-Higgs bundles over $C$ of degree $d$; see \cite{HauselThaddeus04}. However, contrary to the general untwisted character variety, $\MBtw$ is smooth (a significant advantage!).

Surprisingly, Hausel and Rodriguez-Villegas \cite{HauselRodriguez-Villegas2008} in rank 2, and Mellit \cite{Mellit2019} for $\Gl_n$, observed that the cohomology of $\MBtw$ enjoys symmetries \emph{typical} of smooth projective varieties, despite the fact that $\MBtw$ is not projective. They called these symmetries curious hard Lefschetz theorem: there exists a class $\alpha\in H^2(\MBtw)$ which induces the isomorphism
\[
\cup {\alpha^k}\colon \Gr^W_{\dim \MB-2k}H^*(\MBtw)\xrightarrow{\simeq} \Gr^W_{\dim \MB+2k}H^{*+2k}(\MBtw).
\]
Note that, as an immediate consequence of the curious hard Lefschetz theorem, the E-polynomial of $\MBtw$ is palindromic. 

In the attempt to explain the curious hard Lefschetz theorem, de Cataldo, Hausel and Migliorini conjectured the P=W conjecture, and they verified it for rank 2; see \cite{deCataldoHauselMigliorini2012}. This conjecture posits that the non-abelian Hodge correspondence exchanges two filtrations on the cohomology of $\MDoltw$ and $\MBtw$ of very different origin, respectively the perverse Leray filtration \eqref{prop:characperverse} associated to the Hitchin fibration $\chi$ on $\MDoltw$ (the analogue of the map defined in \eqref{Hitchinfibration}) and the weight filtration on $\MBtw$. 
\begin{conj}[P=W conjecture for twisted moduli spaces]
\[P_k H^*(\MDolTW) = \Psi^* W_{2k}H^*(\MBTW).\]
\end{conj}

This suggests that the symmetries of the mixed Hodge structure of the cohomology of twisted character varieties, noted by Hausel and Rodriguez-Villegas, should be understood as a manifestation of the standard relative hard Lefschetz symmetries for the proper map $\chi$ on the Dolbeault side. The latter is an isomorphism between graded pieces of the perverse Leray filtration induced by cup product with a relative $\chi$-ample class $\alpha\in H^2(\MDolTW)$:
\[
\cup {\alpha^k}\colon \Gr^P_{\dim \MDol/2-k}H^*(\MDoltw)\xrightarrow{\simeq} \Gr^W_{\dim \MDol/2+2}H^{*+2k}(\MDoltw);
\]
see for instance \cite[Theorem 2.1.1.(a)]{deCataldoMigliorini05}.

\subsection{PI=WI and the intersection curious hard Lefschetz}
In the untwisted (singular) case, curious hard Lefschetz fails in general, e.g.\ \cite[Remark 7.6]{FelisettiMauri2020}, and the E-polynomial of $\MB(C,G)$ are not palindromic; see for instance \cite[Theorem 1.2]{LogaresMunozNewstead2013}, \cite[Theorem 2]{MartinezMunoz16} or \cite[Theorem 1.3]{BaragliaHekmati17}. In order to restore the symmetries, de Cataldo and Maulik suggested to consider the intersection cohomology of $\MB(C,G)$, and in \cite[Question 4.1.7]{deCataldoMaulik2018} they conjectured:
\begin{conj}[PI=WI conjecture]
\[P_k IH^*(\MDol(C,G)) = \Psi^* W_{2k}IH^*(\MB(C,G)).\]
\end{conj}
As in the twisted case, the PI=WI conjecture and the relative hard Lefschetz theorem for $\chi$ would imply the intersection curious hard Lefschetz theorem.
\begin{conj}[intersection curious hard Lefschetz]\label{conj:intersectioncurioushardLefschetz}
There exists a class $\alpha\in H^2(\MB(C,G))$ which induces the isomorphisms
\[
\cup {\alpha^k}\colon \Gr^W_{\dim \MB-2k}IH^*(\MB(C,G))\xrightarrow{\simeq} \Gr^W_{\dim \MB+2k}IH^{*+2k}(\MB(C,G)).
\]
In particular, the intersection E-polynomial of $\MB(C,G)$ is palindromic.
\end{conj}
In this paper we provide some numerical evidence for \cref{conj:intersectioncurioushardLefschetz}.
\begin{thm}[\cref{cor:palindromic}]\label{thm:palindromicity}
The intersection E-polynomial $IE(\MB(C,G))$ is palindromic for $G=\Gl_2,\Sl_2, \PGl_2$.
\end{thm}

\subsection{PI=WI for $\Sl_2$ is equivalent to PI=WI for $\Gl_2$}
The P=W conjectures for $\Sl_n$ implies the P=W conjectures for $\PGl_n$ and $\Gl_n$; see \cite[\S 3.3]{FelisettiMauri2020}. The converse holds true in the twisted case for $n$ prime by \cite{deCataldoMaulikShen2020}. By \eqref{eq:GLtoSL} and \eqref{eq:PGLtoSL}, this reduction boils down to prove the P=W conjecture for the variant cohomology. The proof in \cref{thm:deCataldoMaulikShen} does not rely on the smoothness of twisted character varieties, and it extends to the singular case verbatim.
\begin{thm}[\cite{deCataldoMaulikShen2020}]\label{thm:deCataldoMaulikShen}
Suppose that
\begin{enumerate}
    \item $ q^{(1-n^2)(2g-2)}IE_{\mathrm{var}}(\MDol(C, \Sl_n); q,q)\eqqcolon q^{(1-n^2)(2g-2)}E(q)$ is palindromic; 
    \item $IE_{\mathrm{var}}(\MB(C, \Sl_n); \sqrt{q},\sqrt{q})=q^{(2-n-n^2)(g-1)}E(q)$.
\end{enumerate}
Set $c_n \coloneqq n(n-1)(g-1)$. Then we have
\[IH^d_{\mathrm{var}}(M(C, \Sl_n)) \simeq \Gr^P_{d-c_n} IH^d_{\mathrm{var}}(\MDol(C, \Sl_n)) \simeq \Gr^W_{2(d-c_n)} IH^d_{\mathrm{var}}(\MB(C, \Sl_n)).\]
\end{thm}
Unfortunately, in the untwisted case the variant intersection E-polynomials are available only in rank 2; see \cref{thm: varinvarIE}.
\begin{cor}\label{cor:PSL2givesGL2vv}
The PI=WI conjecture for $M(C,\Sl_2)$ is equivalent to the PI=WI conjecture for $M(C,\Gl_2)$.
\end{cor}

\subsection{Tautological classes}
In \cite{HauselThaddeus03} Hausel and Thaddeus proved that $H^*(M^{\mathrm{tw}}(C, \Sl_2))^{\Gamma}$ is generated by tautological classes\footnote{The result have been generalised to arbitrary rank in \cite{Markman02}.}. This is an essential ingredient of the proof of the P=W conjecture in the twisted case \cite{deCataldoHauselMigliorini2012} and \cite{deCataldoMaulikShen2019}, and a missing desirable piece of information in the untwisted case. Here we provide a partial result: we show that tautological classes do generate the low-degree intersection cohomology of $M$.

Let $B(C, \Sl_2)$ be the (infinite dimensional and contractible) space of $\Sl_2$-Higgs bundles on $C$ of degree zero, and $B^{ss}(C, \Sl_2)$ be the corresponding locus of semistable Higgs bundles. Let $\mathcal{G}$ be the group of real gauge transformations with fixed determinant acting on this spaces by precomposition, and $\mathcal{G}^{\CC}$ its complexification.

We can identify the classifying space $B\mathcal{G} \simeq B(C, \Sl_2)$ with the space of continuous maps $\Map(C, \operatorname{SU}_2)$. The second Chern class of the tautological (flat) $\operatorname{SU}_2$-bundle $\mathcal{T}$ on $C \times \Map(C, \operatorname{SU}_2)$ admits the K\"{u}nneth decomposition
\[
c_2(\mathcal{T})= \sigma \otimes \alpha + \sum^{2g}_{j=1} e_j \otimes \psi_j + 1 \otimes \beta,
\]
where $\sigma \in H^2(C)$ is the fundamental cohomology class, and $e_1, \ldots, e_{2g}$ is a standard symplectic basis of $H^1(C)$.
Atiyah and Bott showed in \cite{AtiyahBott1983} that the rational cohomology of $B\mathcal{G}$ is freely generated by the tautological classes $\alpha$, $\psi_j$ and $\beta$. That is, $H^*(B\mathcal{G})$ is the tensor product of the polynomial algebra on the classes $\alpha$ and $\beta$ of degree 2 and 4 with an exterior algebra on the classes $\psi_j$ of degree 3,
\begin{equation}\label{eq:AtiyahBott}
    H^*(B\mathcal{G})\simeq \QQ[\alpha, \beta] \otimes \Lambda (\psi_j).
\end{equation}
In particular, the Poincar\'{e} polynomial of the classifying space $B\mathcal{G}$ is
\begin{equation}\label{Poincareclassyfingspace}
    P_t(B\mathcal{G})=\frac{(t^3+1)^{2g}}{(t^2-1)(t^4-1)}.
\end{equation}

Now the non-abelian Hodge correspondence induces the following isomorphism in equivariant cohomology
\begin{equation}\label{eq:isomequivariantcohomology}
    H^*_{\mathcal{G}}(B^{ss}(C, \Sl_2))\simeq H^*_{\Sl_2} (\Hom(\pi_1(C), \Sl_2));
\end{equation}
see \cite[Theorem 1.2]{DaskalopoulosWentworth10}. Together with Kirwan surjectivity \cite[Theorem 1.4]{DWWW11}
\begin{equation}\label{Kirwan}
    H^{*}(B\mathcal{G})\simeq H^{*}_{\mathcal{G}}(B(C, \Sl_2)) \twoheadrightarrow H^{*}_{\mathcal{G}}(B^{ss}(C, \Sl_2))^{\Gamma}, 
\end{equation}
this implies that the $\Gamma$-invariant $\Sl_2$-equivariant cohomology of $\Hom(\pi_1(C), \Sl_2)$ is generated by tautological classes.  
\begin{thm}\label{thm:generation}
$IH^{< 4g-6}(M(C, \Sl_2))$ has a canonical structure of graded ring freely generated by the tautological classes $\alpha$, $\psi_j$, $\beta$ of degree 2, 3, 4 respectively and weight 4. Among the tautological classes, only $\alpha$ is a cohomology class, i.e.\ it is in the image of the natural map $H^*(M(C, \Sl_2)) \to IH^*(M(C, \Sl_2))$.
\end{thm}
\begin{proof}
Since $\Sigma$ has codimension $4g-6$ in $M$, we have 
\begin{equation}\label{eq:smoothlocusIH}IH^{<4g-6}(M(C, \Sl_2))\simeq H^{<4g-6}(M^{\mathrm{sm}}(C, \Sl_2));\end{equation} 
see for instance \cite[Lemma 1]{Durfee95}. In particular, $IH^{<4g-6}(M(C, \Sl_2))$ has a natural structure of graded ring. The open subset of simple representations $\Hom^{s}(\pi_1(C), \Sl_2)$ in $\Hom(\pi_1(C), \Sl_2)$ is a $\PGl_2$-principal bundle over the smooth locus $M^{\mathrm{sm}}_B(C, \Sl_2)$, and so
\begin{equation}\label{eq:smoothlocusstable}
    H^{*}(M^{\mathrm{sm}}(C, \Sl_2)) \simeq H^*_{\Sl_2} (\Hom^{s}(\pi_1(C), \Sl_2)).
\end{equation} 
We claim that the  composition of \eqref{Kirwan}, \eqref{eq:isomequivariantcohomology}, \eqref{eq:simple}, the inverse of \eqref{eq:smoothlocusstable} and \eqref{eq:smoothlocusIH}
\begin{align*}
  H^{<4g-6}(B \mathcal{G}) & \xrightarrow{a} H^{<4g-6}_{\Sl_2} (\Hom(\pi_1(C), \Sl_2)) \\
  & \xrightarrow{b} H^{<4g-6}_{\Sl_2} (\Hom^s(\pi_1(C), \Sl_2))\simeq IH^{< 4g-6}(M(C, \Sl_2))  
\end{align*}
is an isomorphism. Indeed, $a$ is surjective by \eqref{Kirwan} (and \cref{thm: varinvarIE}), and actually bijective since by \cite[Corollary 1.3]{DaskalopoulosUhlenbeck1995} we have 
\[
P^{\Sl_2}_t(\Hom(\pi_1(C), \Sl_2))\coloneqq \sum_d \dim H^d_{\Sl_2}(\Hom(\pi_1(C), \Sl_2))t^d = P_t(B\mathcal{G}) + O(t^{4g-4}).
\]
Further, $b$ is injective by \cref{lem:simplelocus}, and actually bijective due to \eqref{TaylorIPt} and \eqref{Poincareclassyfingspace}. The free generation of $IH^{<4g-6}(M(C, \Sl_2))$ now follows from \eqref{eq:AtiyahBott}. See \cite{Shende17} for the weight of the tautological classes. Finally, $\psi_i$ and $\beta$ are not cohomology classes by \cref{cor:IP-P} and preceding lines.
\end{proof}
\begin{rmk}
\cref{thm:generation} holds for $\PGl_2$, as \[IH^{< 4g-6}(M(C, \Sl_2))=IH^{< 4g-6}(M(C, \Sl_2))^{\Gamma} \simeq IH^{< 4g-6}(M(C, \PGl_2))\] by \cref{thm: varinvarIE}, and so for $\Gl_2$ too. In the latter case however mind that there are additional generators $\epsilon_j$ which are pull-back via the map \eqref{albmap} of a standard basis of $H^1(M(C,\Gl_1))\simeq H^1(C)$. 

As a final remark, note that the proof of \cref{thm:generation} shows the more general statement that Kirwan surjectivity implies the tautological generation of the low-degree intersection cohomology for $M(C, \Sl_n)$. However, this surjectivity is an open problem for $n>2$; cf \cite{CliffNevinsShen18}. 
\end{rmk}
We prove the lemma used in the proof of \cref{thm:generation}.

\begin{lem}\label{lem:simplelocus}
The natural restriction map
\begin{equation}\label{eq:simple}
    H^d_{\Sl_2} (\Hom(\pi_1(C), \Sl_2)) \to H^d_{\Sl_2} (\Hom^s(\pi_1(C), \Sl_2))
\end{equation}
is bijective for $d<4g-7$ and injective for $d=4g-7$. 
\end{lem}
\begin{proof}
Set $c \coloneqq \operatorname{codim}\Sing \Hom(\pi_1(C), \Sl_2)=\dim M - \dim \Sigma + \dim \Stab_\Sigma = 4g-5$, where $\Stab_\Sigma\simeq \Gm$ is the stabiliser of a closed orbit over $\Sigma$. Since $\Hom(\pi_1(C), \Sl_2)$ is a complete intersection (adapt \cite[Theorem 1.2]{Etingof06} or \cite[Proposition 11.3]{Simpson1994}), there exists an isomorphism
\begin{equation}\label{LCI}
H^d_{\Sl_2} (\Hom(\pi_1(C), \Sl_2)) \simeq H^d_{\Sl_2} (\Hom^{\mathrm{sm}}(\pi_1(C), \Sl_2)) \text{ for }d<c-1.
\end{equation}
Indeed, take an approximation $E_k$ of the universal $\Sl_2$-bundle $E\mathrm{SL}_2$, i.e.\ a smooth variety $E_k$ with a free $\Sl_2$-action and such that 
$H^{< k}(X \times_{\Sl_2}E_k)\simeq H^{< k}_{\Sl_2}(X \times_{\Sl_2}E\mathrm{SL}_2)\eqqcolon H^{< k}_{\Sl_2}(X);$
see \cite[Lemma 1.3]{Anderson12}. By Luna slice theorem $X \times_{\Sl_2}E_k$ is a local complete intersection, and the singular locus has again codimension $c$. Then \eqref{LCI} follows from \cite[p.199]{GoreskyMacPherson88}.

Further, the complement of $\Hom^s(\pi_1(C), \Sl_2)$ in the smooth locus has codimension $2g-3$; see for instance \cite[\S 7.2]{Prieto18} where the complement is denoted $\mathfrak{X}^{\rho}_{g}$. Therefore, by the equivariant Thom isomorphism the restriction map
\[H^d_{\Sl_2} (\Hom^{\mathrm{sm}}(\pi_1(C), \Sl_2)) \to H^d_{\Sl_2} (\Hom^{s}(\pi_1(C), \Sl_2))\]
is bijective for $d<4g-7$ and injective for $d=4g-7$. 
\end{proof}

\subsection{P=W vs PI=WI: non-purity of $H^*(\MB)$}\label{P=WPI=WI}
Despite the failure of curious hard Lefschetz, it still makes sense to conjecture $P=W$ phenomena for the ordinary cohomology of $\MB$.

\begin{conj}[P=W conjecture for untwisted character varieties]
\[P_k H^*(\MDol(C,G)) = \Psi^* W_{2k}H^*(\MB(C,G)).\]
\end{conj}

It was proved in \cite[Theorem 6.1]{FelisettiMauri2020} that the PI=WI conjecture for genus 2 and rank 2 implies the P=W conjecture simply by restriction, since $H^*(\MB(C, \Sl_2))$ injects into $IH^*(\MB(C, \Sl_2))$, or equivalently by the purity of $H^*(\MB(C, \Sl_2))$; see \cref{prop:purityandother}. In higher genus the situation is more subtle, as the following theorem shows.

\begin{thm}\label{thm:puritycohomology}
Let $C$ be a curve of genus $g >3$. Then the natural map $H^*(M) \to IH^*(M)$ is not injective. Equivalently, $\MDol$ has no pure cohomology.
\end{thm}
\begin{proof}
Otherwise the polynomial $IP_t(M)-P_t(M)$ would have only positive coefficients, but this is not the case by \cref{cor:IP-P} and preceding lines.
\end{proof}

\begin{rmk}[Torelli group] We propose an alternative proof of \cref{thm:puritycohomology}. The Torelli group is the subgroup of the mapping class group acting trivially on the cohomology of the curve $C$. The Torelli group acts non-trivially on $H^*(M)$ by \cite[Proposition 4.7]{DaskalopoulosWentworth10} (already in degree $6$ by \cite[Theorem 1.1]{CappellLeeMiller00}), but  $IH^{<4g-6}(M)$ is generated by tautological classes due to \cref{thm:generation}, and so the Torelli group acts trivially on $IH^{<4g-6}(M)$ as in \cite[Theorem 2.1.(c)]{CappellLeeMiller00}. Since the natural map $H^*(M) \to IH^*(M)$ is equivariant with respect to the Torelli group, we conclude that it has non-trivial kernel for $g>3$.
\end{rmk}

\subsection{P=W for resolution fails when no symplectic resolution exists}\label{sec:nosymplecticresolution}
In \cite{FelisettiMauri2020} Camilla Felisetti and the author proposed a strong version of PI=WI conjecture, called P=W for resolution, and proved it for character varieties which admits a symplectic resolution. 

\begin{conj}[P=W conjecture for resolution]
There exist resolutions of singularities $f_{\mathrm{Dol}}\colon \tMDolG \to \MDolG$ and $f_{\mathrm{B}}\colon \tMBG \to \MBG$, and a diffeomorphism $\widetilde{\Psi}\colon \tMDolG \to \tMBG$, such that the following square commutes:
\begin{equation}\label{eq:conditiononcommutativity}
\begin{tikzpicture}[baseline=(current  bounding  box.east)]
\node (A) at (0,1.2) {$H^*(\tMDolG, \QQ)$};
\node (B) at (5,1.2) {$H^*(\tMBG, \QQ)$};
\node (C) at (0,0) {$H^*(\MDolG,\QQ)$};
\node (D) at (5,0) {$H^*(\MBG, \QQ)$,};
\draw[<-,thick] (A) -- (C) node [midway,left] {$f_{\mathrm{Dol}}^*$};
\draw[<-,thick] (B) -- (D) node [midway,right] {$f_{\mathrm{B}}^*$};
\draw[<-,thick] (A) -- (B)node [midway,above] {$\widetilde{\Psi}^*$};
\draw[<-,thick] (C) -- (D) node [midway,above] {$\Psi^*$};
\end{tikzpicture}      
\end{equation}
and the lift $\widetilde{\Psi}^*$ of the non-abelian Hodge correspondence ${\Psi}^*$ satisfies the property:
\begin{equation}\label{P=W conjecture for resolution}
    P_kH^*(\tMDolG, \QQ) = \widetilde{\Psi}^* W_{2k}H^*(\tMBG, \QQ).
\end{equation}
\end{conj}

In \cite[Theorem 3.4]{FelisettiMauri2020} Camilla Felisetti and the author proved that resolutions of singularities satisfying \eqref{eq:conditiononcommutativity} do exist, and for instance the Kirwan--O'Grady desingularizations are such. 

In \cref{thm:P=Wfails} below we show however that if $M(C,\Gl_n)$ does not admit a symplectic resolution, no resolution of $M(C,\Gl_n)$ satisfies \eqref{P=W conjecture for resolution}, despite the palindromicity of the E-polynomial of $T_B$; see \cref{thm:EDolT}. A fortiori, the same negative result holds for $G=\Sl_n$.

This means that the hypotheses of \cite[Main Theorem, 3]{Felisetti2018} were optimal for $G=\Gl_n, \Sl_n$: the proof of \cref{thm:P=Wfails} suggests that the semismallness of the desingularization may be a necessary requirement for the P=W conjecture for resolutions to hold for a $G$-character variety with $G$ arbitrary reductive group. This is compatible with the expectation of \cite[\S 4.4]{deCataldoHauselMigliorini2013}. 


\begin{thm}\label{thm:P=Wfails}
Let $M(C, \Gl_n)$ be an untwisted $\Gl_n$-character variety with no symplectic resolution, i.e.\ for $g,n>1$ and $(g,n)\neq (2,2)$. Then the P=W conjecture for any resolution of $M(C, \Gl_n)$ does not hold.
\end{thm}
\begin{proof}
Let $f\colon \widetilde{M} \to M(C, \Gl_n)$ be a resolution of singularities of $M(C, \Gl_n)$ as in \eqref{eq:conditiononcommutativity}, and $E$ be an $f$-exceptional divisor whose image is contained in the singular locus $\Sigma \coloneqq \Sing M(C, \Gl_n)$. Recall that $\chi\colon \MDol(C, \Gl_n) \to \Lambda \coloneqq \bigoplus^n_{i=1} H^0(C, K^{\otimes i}_C)$ is the Hitchin fibration \eqref{Hitchinfibration}. 

The locus $\chi(\Sigma)$ consists of reducible characteristic polynomials, and it has codimension
\[
\frac{1}{2}(\dim M(C, \Gl_n) - \max\{\dim M(C, \Gl_{n_1})+\dim M(C, \Gl_{n_2}): n=n_1+n_2\}) \geq 2.
\]
The last inequality follows for instance from \cite[Lemma 2.2.(2)]{BellamySchedler2019}. In particular, the general affine line in $\Lambda$ avoids $\chi \circ f(E) \subseteq \chi(\Sigma)$. Then by \eqref{prop:characperverse} the Poincar\'{e} dual of $E$ belongs to $P_0 H^2(\widetilde{M})$. However, since $\widetilde{M}_{\mathrm{B}}$ is smooth, $H^2(\widetilde{M}_B)$ has weight not smaller than 2. This contradicts \eqref{P=W conjecture for resolution}. 
\end{proof}

\bibliographystyle{plain}
\bibliography{construction}

\begin{thebibliography}{10}

\bibitem{Anderson12}
D.~Anderson.
\newblock Introduction to equivariant cohomology in algebraic geometry.
\newblock In {\em Contributions to algebraic geometry}, EMS Ser. Congr. Rep.,
  pages 71--92. Eur. Math. Soc., Z\"{u}rich, 2012.

\bibitem{ArbarelloSacca2018}
E.~Arbarello and G.~Sacc\`a.
\newblock Singularities of moduli spaces of sheaves on {K}3 surfaces and
  {N}akajima quiver varieties.
\newblock {\em Adv. Math.}, 329:649--703, 2018.

\bibitem{AtiyahBott1983}
M.~F. Atiyah and R.~Bott.
\newblock The {Y}ang-{M}ills equations over {R}iemann surfaces.
\newblock {\em Philos. Trans. Roy. Soc. London Ser. A}, 308(1505):523--615,
  1983.

\bibitem{BaragliaHekmati17}
D.~Baraglia and P.~Hekmati.
\newblock Arithmetic of singular character varieties and their
  {$E$}-polynomials.
\newblock {\em Proc. Lond. Math. Soc. (3)}, 114(2):293--332, 2017.

\bibitem{BeilinsonBernsteinDeligne1981}
A.~A. Beilinson, J.~N. Bernstein, and P.~Deligne.
\newblock Faisceaux pervers.
\newblock In {\em Analysis and topology on singular spaces, {I}
  ({L}uminy,1981)}, volume 100 of {\em Ast\'{e}risque}, pages 5--171. Soc.
  Math. France, Paris, 1982.

\bibitem{BellamySchedler2019}
G.~{Bellamy} and T.~{Schedler}.
\newblock {Symplectic resolutions of character varieties}.
\newblock {\em arXiv:1909.12545}, 2019.

\bibitem{CappellLeeMiller00}
S.~E. Cappell, R.~Lee, and E.~Y. Miller.
\newblock The action of the {T}orelli group on the homology of representation
  spaces is nontrivial.
\newblock {\em Topology}, 39(4):851--871, 2000.

\bibitem{ChoyKiem07}
J.~Choy and Y.~H. Kiem.
\newblock Nonexistence of a crepant resolution of some moduli spaces of sheaves
  on a {$K3$} surface.
\newblock {\em J. Korean Math. Soc.}, 44(1):35--54, 2007.

\bibitem{CliffNevinsShen18}
E.~{Cliff}, T.~{Nevins}, and S.~{Shen}.
\newblock {On the Kirwan map for moduli of Higgs bundles}.
\newblock {\em arXiv:1808.10311}, 2018.

\bibitem{DWWW11}
G.~Daskalopoulos, J.~Weitsman, R.~A. Wentworth, and G.~Wilkin.
\newblock Morse theory and hyperk\"{a}hler {K}irwan surjectivity for {H}iggs
  bundles.
\newblock {\em J. Differential Geom.}, 87(1):81--115, 2011.

\bibitem{DaskalopoulosUhlenbeck1995}
G.~D. Daskalopoulos and K.~K. Uhlenbeck.
\newblock An application of transversality to the topology of the moduli space
  of stable bundles.
\newblock {\em Topology}, 34(1):203--215, 1995.

\bibitem{DaskalopoulosWentworth10}
G.~D. Daskalopoulos, R.~A. Wentworth, and G.~Wilkin.
\newblock Cohomology of {${\rm SL}(2,\mathbb{C})$} character varieties of
  surface groups and the action of the {T}orelli group.
\newblock {\em Asian J. Math.}, 14(3):359--383, 2010.

\bibitem{deCataldoHauselMigliorini2012}
M.~A. {de}~{Cataldo}, T.~Hausel, and L.~Migliorini.
\newblock Topology of {H}itchin systems and {H}odge theory of character
  varieties: the case {$A_1$}.
\newblock {\em Ann. of Math. (2)}, 175(3):1329--1407, 2012.

\bibitem{deCataldoHauselMigliorini2013}
M.~A. {de}~{Cataldo}, T.~Hausel, and L.~Migliorini.
\newblock Exchange between perverse and weight filtration for the {H}ilbert
  schemes of points of two surfaces.
\newblock {\em J. Singul.}, 7:23--38, 2013.

\bibitem{deCataldoMaulik2018}
M.~A. {de Cataldo} and D.~{Maulik}.
\newblock The perverse filtration for the {H}itchin fibration is locally
  constant.
\newblock {\em arXiv:1808.02235}, 2018.

\bibitem{deCataldoMaulikShen2019}
M.~A. de~Cataldo, D.~Maulik, and J.~{Shen}.
\newblock Hitchin fibrations, abelian surfaces, and the {P=W} conjecture.
\newblock {\em arXiv:1909.11885}, 2019.

\bibitem{deCataldoMaulikShen2020}
M.~A. {de Cataldo}, D.~{Maulik}, and J.~{Shen}.
\newblock {On the P=W conjecture for $\mathrm{SL}_n$}.
\newblock {\em arXiv:2002.03336}, 2020.

\bibitem{deCataldoMigliorini05}
M.~A. de~Cataldo and L.~Migliorini.
\newblock The {H}odge theory of algebraic maps.
\newblock {\em Ann. Sci. \'{E}cole Norm. Sup. (4)}, 38(5):693--750, 2005.

\bibitem{deCataldoMigliorini09}
M.~A. de~Cataldo and L.~Migliorini.
\newblock The decomposition theorem, perverse sheaves and the topology of
  algebraic maps.
\newblock {\em Bull. Amer. Math. Soc. (N.S.)}, 46(4):535--633, 2009.

\bibitem{deCataldoMigliorini2010}
M.~A. de~Cataldo and L.~Migliorini.
\newblock The perverse filtration and the {L}efschetz hyperplane theorem.
\newblock {\em Ann. of Math. (2)}, 171(3):2089--2113, 2010.

\bibitem{Demazure1988}
M.~Demazure.
\newblock Anneaux gradu\'{e}s normaux.
\newblock In {\em Introduction \`a la th\'{e}orie des singularit\'{e}s, {II}},
  volume~37 of {\em Travaux en Cours}, pages 35--68. Hermann, Paris, 1988.

\bibitem{DEL97}
R.~Donagi, L.~Ein, and R.~Lazarsfeld.
\newblock Nilpotent cones and sheaves on {$K3$} surfaces.
\newblock In {\em Birational algebraic geometry ({B}altimore, {MD}, 1996)},
  volume 207 of {\em Contemp. Math.}, pages 51--61. Amer. Math. Soc.,
  Providence, RI, 1997.

\bibitem{Durfee95}
A.~H. Durfee.
\newblock Intersection homology {B}etti numbers.
\newblock {\em Proc. Amer. Math. Soc.}, 123(4):989--993, 1995.

\bibitem{Etingof06}
P.~Etingof.
\newblock Casimirs of the {G}oldman {L}ie algebra of a closed surface.
\newblock {\em Int. Math. Res. Not.}, pages Art. ID 24894, 5, 2006.

\bibitem{Felisetti2018}
C.~{Felisetti}.
\newblock {Intersection cohomology of the moduli space of Higgs bundles on a
  genus 2 curve}.
\newblock {\em arXiv:1805.05815}, 2018.

\bibitem{FelisettiMauri2020}
C.~{Felisetti} and M.~{Mauri}.
\newblock {P=W conjectures for character varieties with symplectic resolution}.
\newblock {\em arXiv:2006.08752}, 2020.

\bibitem{Ginzburg2012}
V.~Ginzburg.
\newblock Lectures on {N}akajima's quiver varieties.
\newblock In {\em Geometric methods in representation theory. {I}}, volume~24
  of {\em S\'{e}min. Congr.}, pages 145--219. Soc. Math. France, Paris, 2012.

\bibitem{GoldmanMillson1988}
W.~M. Goldman and J.~J. Millson.
\newblock The deformation theory of representations of fundamental groups of
  compact {K}\"{a}hler manifolds.
\newblock {\em Inst. Hautes \'{E}tudes Sci. Publ. Math.}, (67):43--96, 1988.

\bibitem{Prieto18}
{A}. {Gonz{\'a}lez-Prieto}.
\newblock {Topological Geometric Invariant Theory and character varieties}.
\newblock {\em arXiv:1807.08540}, 2018.

\bibitem{GM85}
M.~Goresky and R.~MacPherson.
\newblock Lefschetz fixed point theorem for intersection homology.
\newblock {\em Comment. Math. Helv.}, 60(3):366--391, 1985.

\bibitem{GoreskyMacPherson88}
M.~Goresky and R.~MacPherson.
\newblock {\em Stratified {M}orse theory}, volume~14 of {\em Ergebnisse der
  Mathematik und ihrer Grenzgebiete (3)}.
\newblock Springer-Verlag, Berlin, 1988.

\bibitem{Hausel13}
T.~Hausel.
\newblock Global topology of the {H}itchin system.
\newblock In {\em Handbook of moduli. {V}ol. {II}}, volume~25 of {\em Adv.
  Lect. Math. (ALM)}, pages 29--69. Int. Press, Somerville, MA, 2013.

\bibitem{HauselRodriguez-Villegas2008}
T.~Hausel and F.~Rodriguez-Villegas.
\newblock Mixed {H}odge polynomials of character varieties.
\newblock {\em Invent. Math.}, 174(3):555--624, 2008.
\newblock With an appendix by Nicholas M. Katz.

\bibitem{HauselVillegas15}
T.~Hausel and F.~Rodriguez-Villegas.
\newblock Cohomology of large semiprojective hyperk\"{a}hler varieties.
\newblock {\em Ast\'{e}risque}, (370):113--156, 2015.

\bibitem{HauselThaddeus03}
T.~Hausel and M.~Thaddeus.
\newblock Mirror symmetry, {L}anglands duality, and the {H}itchin system.
\newblock {\em Invent. Math.}, 153(1):197--229, 2003.

\bibitem{HauselThaddeus04}
T.~Hausel and M.~Thaddeus.
\newblock Generators for the cohomology ring of the moduli space of rank 2
  {H}iggs bundles.
\newblock {\em Proc. London Math. Soc. (3)}, 88(3):632--658, 2004.

\bibitem{Hitchin1987}
N.~J. Hitchin.
\newblock The self-duality equations on a {R}iemann surface.
\newblock {\em Proc. London Math. Soc. (3)}, 55(1):59--126, 1987.

\bibitem{HuybrechtsLehn2010}
D.~Huybrechts and M.~Lehn.
\newblock {\em The geometry of moduli spaces of sheaves}.
\newblock Cambridge University Press, Cambridge, second edition, 2010.

\bibitem{KaledinLehnSorger2006}
D.~Kaledin, M.~Lehn, and Ch. Sorger.
\newblock Singular symplectic moduli spaces.
\newblock {\em Invent. Math.}, 164(3):591--614, 2006.

\bibitem{Kiem06}
Y.~H. Kiem.
\newblock Intersection cohomology of representation spaces of surface groups.
\newblock {\em Internat. J. Math.}, 17(2):169--182, 2006.

\bibitem{KiemYoo08}
Y.~H. Kiem and S.~B. Yoo.
\newblock The stringy {$E$}-function of the moduli space of {H}iggs bundles
  with trivial determinant.
\newblock {\em Math. Nachr.}, 281(6):817--838, 2008.

\bibitem{Kirwan85}
F.~Kirwan.
\newblock Partial desingularisations of quotients of nonsingular varieties and
  their {B}etti numbers.
\newblock {\em Ann. of Math. (2)}, 122(1):41--85, 1985.

\bibitem{Kirwan86}
F.~Kirwan.
\newblock On the homology of compactifications of moduli spaces of vector
  bundles over a {R}iemann surface.
\newblock {\em Proc. London Math. Soc. (3)}, 53(2):237--266, 1986.

\bibitem{Kirwan87}
F.~Kirwan.
\newblock Rational intersection cohomology of quotient varieties. {II}.
\newblock {\em Invent. Math.}, 90(1):153--167, 1987.

\bibitem{KirwanWoolf06}
F.~Kirwan and J.~Woolf.
\newblock {\em An introduction to intersection homology theory}.
\newblock Chapman \& Hall/CRC, Boca Raton, FL, second edition, 2006.

\bibitem{LogaresMunozNewstead2013}
M.~Logares, V.~Mu\~{n}oz, and P.~Newstead.
\newblock Hodge polynomials of {${\rm SL}(2,\mathbb{C})$}-character varieties
  for curves of small genus.
\newblock {\em Rev. Mat. Complut.}, 26(2):635--703, 2013.

\bibitem{Markman02}
E.~Markman.
\newblock Generators of the cohomology ring of moduli spaces of sheaves on
  symplectic surfaces.
\newblock {\em J. Reine Angew. Math.}, 544:61--82, 2002.

\bibitem{MartinezMunoz16}
J.~Mart\'{\i}nez and V.~Mu\~{n}oz.
\newblock E-polynomials of the {${\rm SL}(2,\mathbb{C})$}-character varieties
  of surface groups.
\newblock {\em Int. Math. Res. Not. IMRN}, (3):926--961, 2016.

\bibitem{Mauri20II}
M.~Mauri.
\newblock Topological mirror symmetry for rank two character varieties of
  surface groups.
\newblock {\em arXiv:2101}, 2021.

\bibitem{Mayrand2018}
M.~{Mayrand}.
\newblock {Local structure of singular hyperkahler quotients}.
\newblock {\em arXiv:1807.05992}, 2018.

\bibitem{Mellit2019}
A.~{Mellit}.
\newblock Cell decompositions of character varieties.
\newblock {\em arXiv:1905.10685}, 2019.

\bibitem{Mereb2015}
M.~Mereb.
\newblock On the {$E$}-polynomials of a family of {$\text{SL}_n$}-character
  varieties.
\newblock {\em Math. Ann.}, 363(3-4):857--892, 2015.

\bibitem{MozgovoyReineke15}
S.~{Mozgovoy} and M.~{Reineke}.
\newblock {Intersection cohomology of moduli spaces of vector bundles over
  curves}.
\newblock {\em arXiv:1512.04076}, 2015.

\bibitem{Nakajima10}
H.~Nakajima.
\newblock {$t$}-analogs of {$q$}-characters of quantum affine algebras of type
  {$E_6,E_7,E_8$}.
\newblock In {\em Representation theory of algebraic groups and quantum
  groups}, volume 284 of {\em Progr. Math.}, pages 257--272.
  Birkh\"{a}user/Springer, New York, 2010.

\bibitem{Namikawa08}
Y.~Namikawa.
\newblock Flops and {P}oisson deformations of symplectic varieties.
\newblock {\em Publ. Res. Inst. Math. Sci.}, 44(2):259--314, 2008.

\bibitem{OGrady99}
K.~G. O'Grady.
\newblock Desingularized moduli spaces of sheaves on a {$K3$}.
\newblock {\em J. Reine Angew. Math.}, 512:49--117, 1999.

\bibitem{OrlikWagreich71}
P.~Orlik and P.~Wagreich.
\newblock Isolated singularities of algebraic surfaces with $\mathbb{C}^*$
  action.
\newblock {\em Ann. of Math. (2)}, 93:205--228, 1971.

\bibitem{PeregoRapagnetta2013}
A.~Perego and A.~Rapagnetta.
\newblock Deformation of the {O}'{G}rady moduli spaces.
\newblock {\em J. Reine Angew. Math.}, 678:1--34, 2013.

\bibitem{PeregoRapagnetta18}
A.~{Perego} and A.~{Rapagnetta}.
\newblock {The moduli spaces of sheaves on K3 surfaces are irreducible
  symplectic varieties}.
\newblock {\em arXiv:1802.01182}, 2018.

\bibitem{PetersSteenbrink2008}
C.~Peters and J.~Steenbrink.
\newblock {\em Mixed {H}odge structures}, volume~52 of {\em Ergebnisse der
  Mathematik und ihrer Grenzgebiete. 3. Folge. A Series of Modern Surveys in
  Mathematics}.
\newblock Springer-Verlag, Berlin, 2008.

\bibitem{Saito89}
M.~Saito.
\newblock Introduction to mixed {H}odge modules.
\newblock Number 179-180, pages 10, 145--162. 1989.
\newblock Actes du Colloque de Th\'{e}orie de Hodge (Luminy, 1987).

\bibitem{Saito90}
M.~Saito.
\newblock Mixed {H}odge modules.
\newblock {\em Publ. Res. Inst. Math. Sci.}, 26(2):221--333, 1990.

\bibitem{Schnell14}
C.~{Schnell}.
\newblock {An overview of Morihiko Saito's theory of mixed Hodge modules}.
\newblock {\em arXiv:1405.3096}, 2014.

\bibitem{Shende17}
V.~Shende.
\newblock The weights of the tautological classes of character varieties.
\newblock {\em Int. Math. Res. Not. IMRN}, (22):6832--6840, 2017.

\bibitem{Simpson1994I}
C.~T. Simpson.
\newblock Moduli of representations of the fundamental group of a smooth
  projective variety. {I}.
\newblock {\em Inst. Hautes \'{E}tudes Sci. Publ. Math.}, (79):47--129, 1994.

\bibitem{Simpson1994}
C.~T. Simpson.
\newblock Moduli of representations of the fundamental group of a smooth
  projective variety. {II}.
\newblock {\em Inst. Hautes \'Etudes Sci. Publ. Math.}, (80):5--79, 1994.

\bibitem{Voisin07I}
C.~Voisin.
\newblock {\em Hodge theory and complex algebraic geometry. {I}}, volume~76 of
  {\em Cambridge Studies in Advanced Mathematics}.
\newblock Cambridge University Press, Cambridge, 2007.

\bibitem{Whitney65}
H.~Whitney.
\newblock Local properties of analytic varieties.
\newblock In {\em Differential and {C}ombinatorial {T}opology ({A} {S}ymposium
  in {H}onor of {M}arston {M}orse)}, pages 205--244. Princeton Univ. Press,
  Princeton, N. J., 1965.

\bibitem{Zhangformal2012}
Z.~Zhang.
\newblock A note on formality and singularities of moduli spaces.
\newblock {\em Mosc. Math. J.}, 12(4):863--879, 885, 2012.

\end{thebibliography}
\end{document}